\documentclass[hidelinks,onefignum,onetabnum,final]{siamart220329}
\usepackage{amsmath,amsfonts,amssymb}
\usepackage{latexsym}
\usepackage{color}
\usepackage{wrapfig}
\usepackage{mathtools} 

\let\eps\varepsilon
\newcommand{\R}{\mathbb R}
\newcommand{\bA}{\mathbf A}
\newcommand{\cA}{\mathcal A}
\newcommand{\hcA}{{\widehat{\mathcal A}}}

\newcommand{\bC}{\mathbf C}

\newcommand{\bS}{\mathbf S}

\newcommand{\bV}{\mathbf V}

\newcommand{\ba}{\mathbf a}

\newcommand{\blf}{\mathbf f}

\newcommand{\be}{\mathbf e}

\newcommand{\bu}{\mathbf u}

\newcommand{\bv}{\mathbf v}

\newcommand{\bx}{\mathbf x}
\newcommand{\by}{\mathbf y}
\newcommand{\bz}{\mathbf z}

\newcommand{\balpha}{\boldsymbol{\alpha}}
\newcommand{\hbalpha}{\widehat{\boldsymbol{\alpha}}}
\newcommand{\bsigma}{\boldsymbol{\sigma}}
\newcommand{\bbeta}{\boldsymbol{\beta}}

\newcommand{\tsigma}{\widetilde{\sigma}}
\newcommand{\halpha}{\widehat{\alpha}}
\newcommand{\bhalpha}{\widehat{\boldsymbol{\alpha}}}

\newcommand{\bphi}{\boldsymbol{\phi}}
\newcommand{\bseta}{\boldsymbol{\eta}}
\newcommand{\bpsi}{\boldsymbol{\psi}}

\newcommand{\bPhi}{\boldsymbol{\Phi}}
\newcommand{\bPsi}{\boldsymbol{\Psi}}
\newcommand{\bTheta}{\boldsymbol{\Theta}}

\newcommand{\tbPhi}{\widetilde{\bPhi}}
\newcommand{\tbPsi}{\widetilde{\bPsi}}
\newcommand{\tbTheta}{\widetilde{\bTheta}}

\newcommand{\rA}{\mathrm A}
\newcommand{\rB}{\mathrm B}
\newcommand{\rC}{\mathrm C}
\newcommand{\rG}{\mathrm G}

\newcommand{\rI}{\mathrm I}

\newcommand{\rP}{\mathrm P}
\newcommand{\rQ}{\mathrm Q}
\newcommand{\rR}{\mathrm R}
\newcommand{\rS}{\mathrm S}
\newcommand{\rU}{\mathrm U}
\newcommand{\rV}{\mathrm V}
\newcommand{\rW}{\mathrm W}
\newcommand{\rY}{\mathrm Y}
\newcommand{\rZ}{\mathrm Z}
\newcommand{\rF}{\mathrm F}

\newtheorem{remark}{Remark}

\begin{document}
\title{Tensorial parametric  model order reduction of nonlinear dynamical systems}
\author{Alexander V. Mamonov\thanks{Department of Mathematics, University of Houston, 
Houston, Texas 77204 (avmamonov@uh.edu).} \and
Maxim A. Olshanskii\thanks{Department of Mathematics, University of Houston, 
Houston, Texas 77204 (maolshanskiy@uh.edu).}
}
\maketitle

\begin{abstract}
For a nonlinear dynamical system that depends on parameters, the paper introduces a novel tensorial reduced-order model (TROM). The reduced model is projection-based, and for systems with no parameters involved, it resembles proper orthogonal decomposition (POD) combined with the discrete empirical interpolation method (DEIM). For parametric systems, TROM employs low-rank tensor approximations in place of truncated SVD, a key dimension-reduction technique in POD with DEIM. Three popular low-rank tensor compression formats are considered for this purpose: canonical polyadic, Tucker, and tensor train. The use of multilinear algebra tools allows the incorporation of information about the parameter dependence of the system into the reduced model and leads to a POD-DEIM type ROM that (i) is parameter-specific (localized) and predicts the system dynamics for out-of-training set (unseen) parameter values, (ii) mitigates the adverse effects of high parameter space dimension, (iii) has online computational costs that depend only on tensor compression ranks but not on the full-order model size, and (iv) achieves lower reduced space dimensions compared to the conventional POD-DEIM ROM. The paper explains the method, analyzes its prediction power, and assesses its performance for two specific parameter-dependent nonlinear dynamical systems.
\end{abstract}


\section{Introduction}

The numerical solution of parametric dynamical systems is a common problem in areas such as numerical optimal control, shape optimization, inverse modeling, and uncertainty quantification. 
If the system is described by a set of evolutionary nonlinear partial differential equations (PDEs), then a straightforward approach that involves repeatedly solving a fully resolved discrete model for various parameter values can result in overwhelming computational costs. Reduced-order models (ROMs) offer a possibility to alleviate these costs by replacing the fully resolved model, also referred to as the full-order model (FOM), with a low-dimensional surrogate model~\cite{antoulas2000survey,gugercin2004survey}.

Reduced-order modeling for parametric dynamical systems has already attracted considerable attention, as seen in works such as \cite{benner2015survey,hesthaven2016certified, bui2008model, baur2011interpolatory,
	benner2014robust,brunton2016discovering}. This paper contributes to the topic with a novel projection-based ROM that extends the ideas of proper orthogonal decomposition (POD) and discrete empirical interpolation method (DEIM) ROMs \cite{lumley1967structure,sirovich1987turbulence,chaturantabut2010nonlinear}  to parametric systems using concepts and tools from tensor algebra. We refer to this approach as 'tensorial reduced-order modeling' (TROM).

In general, a projection-based ROM constructs the surrogate model by projecting the high-fidelity FOM onto a low-dimensional problem-dependent vector space~\cite{benner2015survey}. This space is computed from information provided by FOM solutions sampled at specific time instances and parameter values, often referred to as solution snapshots. When dealing with multiple and varying parameters, it can be challenging, or even impossible, to build a universal low-dimensional space that adequately represents solutions for all times and parameters of interest while still achieving sufficient order reduction. TROM addresses this problem by using a 'low-rank tensor decomposition' (LRTD) in place of truncated singular value decomposition (SVD), a key dimension-reduction technique in both POD and DEIM. Similar to SVD, LRTD provides an orthogonal basis for a 'universal' reduced-order space that can be sufficiently large to approximate the space of all observed snapshots. However, tensor low-rank representations also preserve information about the parameter dependence in this universal space, something that SVD/POD fail to offer. For any incoming parameter value, not necessarily from the training set, TROM uses this additional information to find a 'parameter-specific' local reduced space, which is a low-dimensional subspace of the universal ROM space.

According to the outline above, dimension reduction in TROM is a two-stage process. In the first, 'offline' stage, two approximate LRTDs are computed: one for the tensor of solution snapshots and another for the tensor of snapshots of the nonlinear term of the dynamical system. The second LRTD is required for a hyper-reduction DEIM-type method. The system is then projected onto the universal space provided by the first LRTD, and DEIM is performed using the universal space of the second LRTD. The resulting significantly reduced but still relatively large projected system is then passed to the next stage along with certain information about both LRTDs needed to compute the local reduced bases. In the second, 'online' stage, for a given incoming vector of parameters, TROM computes bases for the local parameter-specific subspaces. These orthogonal bases are represented by their coordinates in the universal spaces. This allows for easy projection of the system onto the local reduced subspace and for performing a second (local) step of hyper-reduction. As demonstrated below, computation of local reduced bases and projection onto the local subspaces during the online stage involves operations only with low-dimensional matrices and vectors, making it fast. The distinguishing features of TROM are as follows: (i) it finds 'parameter-specific' (local) reduced spaces during the online stage; (ii) the additional online costs are small and depend only on tensor ranks, not on the FOM resolution; (iii) depending on the low-rank tensor compression format used, the adverse effects of parameter space dimension can be mitigated; and (iv) reduced space dimensions are lower compared to the traditional POD--DEIM ROM for the same or better accuracy.

The concept of tensorial ROM was introduced recently in~\cite{mamonov2022interpolatory}. That paper explained the computation of universal and local projection spaces for three popular low-rank tensor formats (canonical polyadic, Tucker, and tensor train) and applied the method to two parameterized linear systems. This paper extends TROM for reduced-order modeling of nonlinear systems, which requires the application of a hyper-reduction technique, and introduces the concept of a tensor two-stage DEIM. We show that finding a local parameter-dependent representation of the nonlinear terms can be done in two stages, including LRTD in the offline stage and low-dimensional computations in the online stage. We provide an interpolation estimate for tensorial DEIM in terms of tensor decomposition accuracy, interpolation bounds in the parameter domain, and singular values of some local matrices.

The literature on tensor methods in reduced-order modeling of dynamical systems is rather limited. In addition to~\cite{mamonov2022interpolatory}, we mention two papers \cite{nouy2015low,nouy2017low} that review tensor compressed formats and discuss their possible use for sparse function representation and reduced-order modeling, as well as a series of publications on the tensorization of algebraic systems resulting from the stochastic and parametric Galerkin finite element method, see, e.g.,~\cite{benner2015low,benner2016low,benner2017solving,lee2019low, kressner2011low}. In~\cite{kastian2020two}, a POD-ROM was combined with an LRTD of a mapping from parameter space onto an output domain. Different approaches to making projection-based ROMs parameter-specific can be found in \cite{eftang2010hp,eftang2011hp,amsallem2012nonlinear,amsallem2008interpolation,son2013real}.

The rest of the paper is organized into six sections. Section~\ref{s:setup} sets up the problem of interest. Section~\ref{s:pod} summarizes the standard POD-DEIM ROM. Section~\ref{s:prelim} introduces the necessary tensor algebra preliminaries and reviews the concept of tensor rank. Section~\ref{s:TROM} explains TROM. Section~\ref{s:analysis} addresses the analysis of TROM, including the representation capacity of local reduced spaces and the interpolation property of tensorial DEIM. Finally, Section~\ref{s:num} assesses the performance of TROM for two examples of parameterized dynamical systems and compares it to that of the standard POD--DEIM ROM.

\textbf{Notation conventions.}
The TROM and its analysis involve vectors, matrices, and tensors tailored to the full model, the reduced model, and some intermediate constructions. To assist the reader in navigating through the paper, we follow several notation conventions:we use lowercase letters for scalars,  bold lowercase letters for vectors, upright capital letters for matrices,  and all tensors will be denoted with bold uppercase letters (e.g. $v$ is a scalar, $\bv$ is a vector, $\rV$ is a matrix, and $\bV$ would be a tensor).
For dimensions associated with the full-order model, we use uppercase Latin letters like $N$, $M$, $K$, $K_1$, and so on. To denote low-rank approximations of full-order tensors and related quantities, we use the tilde symbol (~). For example, if $\mathbf{\Phi}$ is a full-order tensor, then $\widetilde{\mathbf{\Phi}}$ represents its low-rank approximation. We may also use $\widetilde{N}$, $\widetilde{M}$, or $\widetilde{R}_1$, and so forth to denote tensor ranks.
Vector spaces are denoted by uppercase Latin letters such as $V$, $U$, $Y$, and so on (Note: These should not be confused with upright capitals $\rV$, $\rU$, $\rY$ used for matrices).
We reserve the letter 'n' to represent the final reduced dimension of a ROM. When it's necessary to distinguish between the final reduced dimensions of the ROM projection and DEIM interpolation spaces, we use $n_{\Phi}$ and $n_{\Psi}$, respectively.

\section{Problem formulation}
\label{s:setup} 

The TROM framework we develop applies to steady-state and evolutionary systems arising from 
PDEs depending on parameters. We formulate the problem in the form of a general non-linear 
dynamical system. Specifically, for a given $\balpha = (\alpha_1,\dots,\alpha_D)$ from the parameter domain 
$\cA \subset \mathbb{R}^D$, find the trajectory $\bu = \bu(t, \balpha) : [0,T) \to \mathbb{R}^M$ solving
\begin{equation}
\label{eqn:GenericPDE}
\bu_t = \rA_{\balpha}\bu+ \blf_{\balpha}(t, \bu),  \quad t \in (0,T), \quad \text{and}~ \bu|_{t=0} = \bu_0,
\end{equation}
with parameter dependent matrix $\rA_{\balpha}\in\mathbb{R}^{M\times M}$, 
continuous flow field $\blf_{\balpha}:(0,T) \times \mathbb{R}^M \to \mathbb{R}^M$, and an initial condition $\bu_0$.
We assume that the unique solution exists on $(0,T)$ for all $\balpha\in\cA$. 

One can think about \eqref{eqn:GenericPDE} as a system of ODEs resulting from a spatial discretization of a (nonlinear) parabolic problem, where the coefficients, boundary conditions, or the computational domain (through a mapping into a reference domain) are parameterized by $\balpha$.

Our focus is on a projection-based ROM, where, for any given $\balpha\in\cA$, we seek an approximation to $\bu(\balpha)$ by solving a lower-dimensional system of equations obtained by projecting \eqref{eqn:GenericPDE} onto a reduced space, known as the ROM space. For effective order reduction, it is essential that the ROM space is both problem-dependent and specific to $\balpha$.

Among the projection-based approaches for model reduction in time-dependent differential equations, one of the most common techniques is proper orthogonal decomposition enhanced with the discrete empirical interpolation method (POD--DEIM) to handle nonlinear terms~\cite{kerschen2005method, rathinam2003new, liang2002properi, liang2002properii}. We provide an outline of the POD--DEIM ROM below for reference and for the purpose of comparison in Section~\ref{s:num}. Additionally, reviewing the POD--DEIM ROM is instructive since the Tensorial Reduced Order Modeling (TROM) can be viewed as a natural extension of POD--DEIM to parametric problems.

\section{Model reduction via POD--DEIM}
\label{s:pod}

Here we recap the conventional POD--DEIM ROM for non-linear systems
adapted for the parametric case. Consider a training set of $K$ parameters sampled from the parameter 
domain, ${\hcA} := \{\widehat\balpha_1, \dots, \widehat\balpha_K \} \subset \cA$. 
Hereafter we use hats to denote parameters from the training set $\hcA$.  
%
At the first, offline stage of POD--DEIM, one computes through FOM numerical simulations a collection 
of solution snapshots 
\begin{equation}
\bphi_j(\widehat\balpha_k) = \bu(t_j,\widehat\balpha_k) \in \mathbb{R}^M, 
\quad j = 1,\ldots, N, \quad k = 1,\ldots,K,
\end{equation}
and non-linear term snapshots  
\begin{equation}
\bpsi_j(\widehat\balpha_k) = \blf_{\widehat\balpha_k} (t_j, \bu(t_j, \widehat\balpha_k)) 
\in \mathbb{R}^M, \quad j = 1,\ldots, N, \quad k = 1,\ldots,K,
\end{equation}
further referred to as $\bu$- and $\blf$-snapshots, respectively, at times $0\le t_1,\dots,t_N\le T$, 
and for $\widehat\balpha_k$ from the training set ${\hcA}$. For a desired reduced space dimension 
$n \ll M$, one computes the reduced space basis $\{ \bu_i^{\rm pod}\}_{i=1}^{n} \subset \mathbb{R}^M$, 
referred to  as the \emph{POD basis}, such that the projection subspace 
$\mbox{span} \big\{ \bu_1^{\rm pod},\dots, \bu_n^{\rm pod}\big\}$ approximates 
the space spanned by all $\bu$-snapshots 
in the best possible way. 
This is achieved by assembling the matrix of all $\bu$-snapshots
\begin{equation}
\label{eqn:Phi}
\Phi_{\rm pod} = [\bphi_1(\balpha_1), \ldots, \bphi_N(\balpha_1), \ldots,
\bphi_1(\balpha_K), \ldots,\bphi_N(\balpha_K),] \in \R^{M \times N K}
\end{equation}
and computing its SVD 
\begin{equation}
	\label{eqn:SVDa}
\Phi_{\rm pod} = \rU \Sigma \rV^T.
\end{equation}
Then, the POD reduced basis vectors $\bu_i^{\rm pod}$, $i=1,\dots,n$, are taken to be the first $n$ 
left singular vectors of $\Phi_{\rm pod}$, i.e., the first $n$ columns of $\rU$ or in Matlab notation 
$\rU_{\rm pod}=\rU_{:, 1:n}$.

At the second, online stage, the POD--ROM solution $\bu^{\rm rom}$ is found through its vector of 
coordinates $\bbeta$ in the space $\mbox{range}(\rU_{\rm pod})$, 
i.e., $\bu^{\rm rom}=\rU_{\rm pod}\bbeta$, 
which solve the projected system
\begin{equation}
\label{eqn:genericROM}
\bbeta_t = \rU_{\rm pod}^T\rA_{\balpha}\rU_{\rm pod}\bbeta
+ \rU_{\rm pod}^T\blf_{\balpha}(t, \rU_{\rm pod} \bbeta),  \quad t \in (0,T), 
\mbox{~~and~~} \bbeta|_{t=0} = \rU_{\rm pod}^T\bu_0.
\end{equation}

While the pre-computation of the projected matrix $\mathbf{U}_{\text{pod}}^T\mathbf{A}_{\balpha}\mathbf{U}_{\text{pod}}$ during the offline stage enables fast matrix-vector multiplications for evaluating \eqref{eqn:genericROM}, efficiently evaluating the nonlinear term in \eqref{eqn:genericROM} during the online stage is generally challenging. The discrete empirical interpolation method (DEIM) addresses this challenge effectively. In this widely-used hyper-reduction technique, the nonlinear term is approximated within a lower-dimensional subspace of $\text{span}\{\mathbf{\psi}_j(\balpha_k)\}_{j=1,\ldots, N,k=1,\dots,K}$, the space spanned by  all FOM $\blf$-snapshots. SVD is employed to find the basis $\{\mathbf{y}_i^{\text{pod}}\}_{i=1}^{n} \subset \mathbb{R}^M$ for this subspace, where $\mathbf{y}_i^{\text{pod}}$ represents the $i$-th left singular vector of the matrix:
\begin{equation}
\label{eqn:Psi}
\Psi_{\rm pod} = [\bpsi_1(\balpha_1), \ldots, \bpsi_N(\balpha_1), \ldots,\bpsi_1(\balpha_K), 
\ldots,\bpsi_N(\balpha_K)] \in \R^{M\times NK},
 \end{equation}
comprising all $\blf$-snapshots. To ease notation, also take $n$ first left singular vectors of 
\eqref{eqn:Psi} to form $\rY_{\rm pod}=[ \by_1^{\rm pod},\dots, \by_{n}^{\rm pod}]$.
For stability and accuracy considerations, it's possible for $\mathbf{Y}_{\text{pod}}$ and $\mathbf{U}_{\text{pod}}$ to contain different numbers of vectors. In such cases, we refer to these dimensions as $n_{\Psi}$ and $n_{\Phi}$, respectively.
We refer to $\mbox{range}(\rY_{\rm pod})$ as the POD--DEIM reduced $\blf$-space and the columns of 
$\rY_{\rm pod}$ as the POD--DEIM $\blf$-reduced basis. Then, DEIM approximates the nonlinear term of 
\eqref{eqn:GenericPDE} via
\begin{equation}
\blf_{\balpha}(t,\bu)\approx \rY_{\rm pod} ( \rP^T  \rY_{\rm pod})^{-1} 
 \rP^T  \blf_{\balpha}(t,\bu),
\end{equation}  
where the 'selection' matrix is defined as:
\begin{equation}
	\mathbf{P} = \mathbf{P}(\bseta) := [\mathbf{e}_{\eta_1}, \ldots, \mathbf{e}_{\eta_n}] \in \mathbb{R}^{M \times n},
\end{equation}
This matrix, $\mathbf{P}$, is constructed so that for any $\blf\in\mathbb{R}^M$, the vector $\mathbf{P}^T \blf$ contains $n$ entries selected from $\blf$ with indices $\bseta = [\eta_1, \ldots, \eta_n]^T \in \mathbb{R}^n$. DEIM determines $\bseta$ entirely based on the information within $\mathbf{Y}_{\text{pod}}$ using a greedy algorithm, as detailed in~\cite{chaturantabut2010nonlinear}.
 
The singular values of $\Phi_{\rm pod}$ and $\Psi_{\rm pod}$  provide information 
about the representation power of $\mbox{span} \big\{ \bu_1^{\rm pod},\dots, \bu_n^{\rm pod}\big\}$ 
and $\mbox{span} \big\{ \by_1^{\rm pod},\dots, \by_n^{\rm pod}\big\}$, respectively. 
In particular, it holds: 
\begin{equation}
\label{eqn:phipsibound}
\mbox{\small $\displaystyle  \sum_{k=1}^{K}\sum_{i=1}^{N}$ } \bigg\| \bphi_i(\balpha_k) - 
\mbox{\small $\displaystyle  \sum_{j=1}^{n}$ } \left\langle \bphi_i(\balpha_k), \bu_j^{\rm pod} \right\rangle \bu_j^{\rm pod} 
\bigg\|^2_{\ell^2} \le 
\mbox{\small $\displaystyle  \sum_{j=n+1}^{NK} $ }
\sigma_i^2 (\Phi_{\rm pod}),
\end{equation}
for representation of the solution states, and similarly for $\by^{\rm pod}_j$ and $\blf$-snapshots 
$\bpsi_i(\balpha_k)$.

Summarizing, the POD--DEIM ROM of \eqref{eqn:GenericPDE} takes the form 
\begin{equation}
\bbeta_t = \rU_{\rm pod}^T\rA_{\balpha}\rU_{\rm pod}\bbeta + 
(\rU_{\rm pod}^T \rY_{\rm pod}) ( \rP^T  \rY_{\rm pod})^{-1}  \rP^T  
\blf_{\balpha}(t, \rU_{\rm pod}\bbeta),  \quad t \in (0,T), 
\end{equation}
with the initial condition
$\bbeta|_{t=0} = \rU_{\rm pod}^T\bu_0$, 
for $\bbeta(t):[0,T]\to \mathbb{R}^n$ so that $\bu(t)$ is approximated by 
$\bu^{\rm rom}(t):[0,T]\to \mbox{span} \big\{ \bu_1^{\rm pod},\dots, \bu_n^{\rm pod}\big\}$, 
where $\bu^{\rm rom}(t)=\rU_{\rm pod}\bbeta(t)$. 

A key requirement for the efficient evaluation of the nonlinear term is that for a fixed $t$ and 
$\bu\in\mathbb{R}^M$, an arbitrary given entry $f_i$ of vector $\blf_{\balpha}$ can be computed quickly, 
with costs independent of dimensions $M$ and $N$. To meet this requirement, we assume that each 
$f_i$ depends on a few entries of $\bu$, i.e.,
$
f_i(t,\bu) = f_i(t, u_{p_1(i)}, \ldots, u_{p_C(i)})
$, 
with $C$ independent on $M$ and $N$. 

Please note that the POD--DEIM reduced bases capture cumulative rather than localized information about the dependence of $\bu$- and $\blf$-snapshots on $\balpha$. Without this parameter-specificity, both bases may lack robustness for parameter values outside the training set, and this limitation can even apply to in-sample parameters if the reduced dimension is not sufficiently high. In other words, POD--DEIM bases might not perform well away from the reference FOM simulations. This poses a significant challenge when applying POD-based ROMs in tasks like inverse modeling. To address this challenge, we introduce tensor techniques with the aim to preserve the information  about parameter dependence in lower dimensional spaces and next to benefit from it at the online stage. 
We continue with  some preliminaries from multi-linear algebra.

\section{Multi-linear algebra preliminaries and tensor decompositions}
\label{s:prelim}

Assume that the parameter domain $\cA$ is the $D$-dimensional box 
\begin{equation}
	\cA = {\textstyle \bigotimes\limits_{i=1}^D} [\alpha_i^{\min}, \alpha_i^{\max}].
	\label{eqn:box}
\end{equation}  
Also, let the training set $\hcA$ be a Cartesian grid: distribute $K_i$ nodes $\{\halpha_i^j\}_{j=1,\dots,K_i}$ 
within each of the intervals $[\alpha_i^{\min}, \alpha_i^{\max}]$ in \eqref{eqn:box} for $i=1,\dots,D$, 
and let
\begin{equation}
\label{eqn:grid}
	\hcA = \left\{ \bhalpha =(\halpha_1,\dots,\halpha_D)^T\,:\,
	\halpha_i \in \{\halpha_i^j\}_{j=1,\dots,K_i}, ~ i = 1,\dots,D \right\}.
\end{equation} 
The cardinality of $\hcA$ is obviously $K = \prod_{i=1}^{D} K_i$.

Given the structure \eqref{eqn:grid} of the training set, the FOM solution $\bu$- and $\blf$-snapshots 
are naturally organized in the multi-dimensional arrays
\begin{equation}
(\bPhi)_{:,j_1,\dots,j_D,k} = \bphi_k(\halpha_1^{j_1},\dots,\halpha_D^{j_D}), \quad 
(\bPsi)_{:,j_1,\dots,j_D,k} = \bpsi_k(\halpha_1^{j_1},\dots,\halpha_D^{j_D}),
\label{eqn:snapmulti}
\end{equation}
which are tensors of order $D+2$ and size $M\times K_1\times\dots\times K_D\times N$, i.e.,
$j_i = 1,\ldots, K_i$, $i=1,\ldots,D$, $k=1,\ldots,N$.  We reserve the first and last indices of $\bPhi$, $\bPsi$ for dimensions corresponding to the spatial and temporal resolution, respectively.

Throughout the rest of this  section, $\bTheta$ is a generic tensor of the same order and size as the $\bu$- and $\blf$-snapshots tensors.
Unfolding of  $\bTheta$ is reordering  of its elements into a matrix. If all 1st-mode fibers of $\bTheta$, i.e. all vectors $(\bTheta)_{:,j_1,\dots,j_D,k}\in\R^M$,  are organized into columns of a  $M\times NK$ matrix, we get the \emph{1st-mode unfolding matrix}, denoted by $\Theta_{(1)}$. A particular ordering of the columns in   $\Theta_{(1)}$ is not important for the purposes of this paper. Thus, $\Phi_{\rm pod}$ and $\Psi_{\rm pod}$ are 1st-mode unfolding matrices of tensors 
$\bPhi$ and  $\bPsi$.
We seek to replace the (truncated) SVDs of $\Phi_{\rm pod}$ and $\Psi_{\text{deim}}$ with low-rank approximations of $\bPhi$ and $\bPsi$ directly in tensor format. 

Unlike the matrix case, the notion of tensor rank is ambiguous. The problem of defining a tensor rank(s) 
has been extensively addressed in the literature; see e.g.~\cite{hackbusch2012tensor}.  For our purpose we choose three tensor formats which lead to different definitions of a rank and offer three compressed
tensor representations or LRTD. These three formats: canonical polyadic (CP), Tucker (a.k.a high-order
singular value decomposition, HOSVD), and tensor train (TT), are recalled below.  

In the CP format ~\cite{hitchcock1927expression,carroll1970analysis,kiers2000towards,ReviewTensor}, 
one represents  a tensor $\bTheta$ by the sum of $R$  outer products of $D+2$ vectors 
$\bu^r\in\R^M$, $\bsigma^{r}_i \in \R^{K_i}$, $i=1,\dots,D$, and $\bv^r \in \R^N$,
\begin{equation}
	\label{eqn:CPv}
	\bTheta \approx \widetilde{\bTheta} = 
	\mbox{\small $\displaystyle \sum_{r=1}^{\widetilde R}$} \bu^r \circ \bsigma^r_1 \circ \dots \circ \bsigma^r_D \circ \bv^r.
\end{equation}
Here $\widetilde R$ is the so-called CP-rank and  an outer product of $d$ vectors $\ba^{(j)}\in\mathbb{R}^{N_j}$ is defined as an $N_1\times\dots\times N_d$ rank one tensor $\bA=\ba^{(1)}\circ \dots \circ\ba^{(d)}$ with entries $\bA_{i_1\dots i_d}=a^{(1)}_{i_1}\dots a^{(d)}_{i_d}$.

The HOSVD represents a tensor $\bTheta$ in the Tucker format~\cite{tucker1966some,ReviewTensor}:
\begin{equation}
	\label{eqn:TDv}
	\bTheta \approx \widetilde{\bTheta} = 
\mbox{\small $\displaystyle 
	\sum_{j = 1}^{\widetilde{M}}
	\sum_{q_1 = 1}^{\widetilde{K}_1}\dots
	\sum_{q_D = 1}^{\widetilde{K}_D}
	\sum_{k = 1}^{\widetilde{N}}$
}
	(\bC)_{j, q_1, \dots, q_D, k} \bu^j \circ \bsigma^{q_1}_1 \circ \dots \circ \bsigma^{q_D}_D \circ \bv^k,
\end{equation}
with a core tensor $\bC$ and vectors $\bu^j \in \R^M$, $\bsigma_i^{q_i} \in \R^{K_i}$, and $\bv^k \in \R^N$. 
The sizes of core tensor in all dimensions, i.e., 
$\widetilde{M}$, $\widetilde{K}_1$, $\ldots$, $\widetilde{K}_D$ and $\widetilde{N}$, 
are referred to as Tucker ranks of $\widetilde{\bTheta}$. 

Finally, the  tensor train  decomposition~\cite{TT1} represents a tensor  in the TT-format:
\begin{equation}
	\label{eqn:TTv}
	\bTheta \approx \widetilde{\bTheta} =
	\mbox{\small $\displaystyle \sum_{j_1=1}^{\widetilde R_1}\dots
	\sum_{j_{D+1}=1}^{\widetilde R_{D+1}}$}
	\bu^{j_1} \circ \bsigma^{j_1, j_2}_1 \circ \dots \circ \bsigma^{j_D, j_{D+1}}_D \circ \bv^{j_{D+1}},
\end{equation}
with $\bu^{j_1} \in \R^M$, $\bsigma^{j_i, j_{i+1}}_i \in \R^{K_i}$, and $\bv^{j_{D+1}} \in \R^N$,
where the positive integers $\widetilde R_i$ are referred to as the \emph{compression ranks} (or TT-ranks) 
of the decomposition. For higher order tensors the TT format is in general more efficient compared to HOSVD. 
This may be beneficial for larger $D$. 
Note that unlike CP or HOSVD formats, compression ranks of TT decomposition may depend on the order
in which the snapshots are organized in tensors $\bPhi$ and $\bPsi$. Throughout this paper we use the ordering
as in \eqref{eqn:snapmulti}. However, a different order may decrease the compression ranks further. 

All three decompositions can be viewed as extensions of SVD to multi-dimensional arrays but having different 
numerical and compression properties. In particular, finding the best approximation of 
tensor by a fixed-ranks tensor in Tucker and TT format is a well-posed problem with constructive algorithms 
known to deliver quasi-optimal solutions~\cite{de2000multilinear,TT1}. Furthermore, using these algorithms 
based on truncated SVD for a sequence of unfolding matrices, one may find $\widetilde{\bTheta}$ 
(in Tucker or TT format) that satisfies
\begin{equation}
	\label{eqn:TensApprox}
	\big\| \widetilde{\bTheta} - \bTheta \big\|_F \le {\eps}\big\|\bTheta \big\|_F
\end{equation}
for given ${\eps} > 0$. Corresponding Tucker or TT ranks are then recovered in the course of factorization. 
Here and further, $\|\bTheta \big\|_F$ denotes the tensor Frobenius norm, which is simply the  square root of the sum of the squares of all  entries of $\bTheta$.

The $k$-mode tensor-vector product $\bTheta \times_k \ba$ of a tensor 
$\bTheta \in \R^{N_1 \times \dots \times N_m}$ of order $m$ and a vector $\ba \in \mathbb{R}^{N_k}$
is a tensor of order $m-1$ and size 
$N_1 \times \dots \times N_{k-1} \times N_{k+1} \times \dots \times N_m$:
\begin{equation}
	(\bTheta \times_k \ba)_{j_1,\dots, j_{k-1},j_{k+1},\dots, j_{m}}=
	\mbox{\small $\displaystyle\sum_{j_k = 1}^{N_k}$} \bTheta_{j_1, \dots, j_{m}} a_{j_k}.
	\label{eqn:kmodeprod}
\end{equation}
Analogously, the $k$-mode tensor-matrix product $\bTheta \times_k \rA$ of a tensor 
$\bTheta \in \R^{N_1 \times \dots \times N_m}$ and a matrix $\rA \in \mathbb{R}^{J \times N_k}$
is a tensor of order $m$ and size $N_1 \times \dots \times N_{k-1}, J, N_{k+1} \times \dots \times N_m$:
\begin{equation}
	(\bTheta \times_k \rA)_{j_1,\dots, j_{k-1},i,j_{k+1},\dots, j_{m}}=
	\mbox{\small $\displaystyle\sum_{j_k = 1}^{N_k}$} \bTheta_{j_1, \dots, j_{m}} a_{ij_k}.
	\label{eqn:kmodeprodA}
\end{equation}

In what follows we assume that $\widetilde\bPhi$ and  $\widetilde\bPsi$ are compressed representations 
of $\bu$- and $\blf$-snapshot tensors $\bPhi$ and $\bPsi$, respectively, in one of the formats discussed above. 
To minimize notation burden, we assume they satisfy \eqref{eqn:TensApprox} with some $\eps$ which is the 
same for both $\bPhi$ and $\bPsi$ compression.  
\newpage

\section{Tensorial ROM for nonlinear dynamical systems} 
\label{s:TROM} 
\subsection{Universal and local reduced spaces}
Each of the two stages of the TROM algorithm is associated with distinct reduced spaces. The reduced spaces computed during the offline stage are referred to as the 'universal' reduced spaces, while the 'local' reduced spaces are employed in the online stage. We will describe both types of reduced spaces below.

\emph{Universal reduced spaces are the spans of all 1st-mode fibers of the compressed tensors}, i.e.,
\begin{equation}
\widetilde{U} = \mbox{range}\big(\widetilde\Phi_{(1)}\big)~~ \mbox{ and } ~~
\widetilde{Y} = \mbox{range}\big(\widetilde\Psi_{(1)}\big).
\end{equation}
The dimension of $\widetilde{U}$ is equal to the first Tucker or TT rank of $\widetilde\bPhi$ 
(if Tucker or TT formats are used) and it does not exceed $R$ for the 
CP compression format. We also denote by $\rU$ and $\rY$ the matrices with columns that form 
orthonormal bases for $\widetilde{U}$ and $\widetilde{Y}$, respectively. 

The universal spaces represent all observed $\bu$- and $\blf$-snapshots up to the tensor compression accuracy. 
Indeed, let $\rU=[\bu_1,\dots,\bu_{\widetilde N}]$, where  $\widetilde N=\mbox{dim}(\widetilde{U})$, then it holds
{\small
\begin{equation}
	\label{eqn:tphibound}
\begin{split}
	\mbox{\small $\displaystyle \sum_{k=1}^{K}\sum_{i=1}^{N}$}\left\| \bphi_i(\balpha_k) - 
	\mbox{\small $\displaystyle \sum_{j=1}^{\widetilde N}$}\langle \bphi_i(\balpha_k), \bu_j \rangle \bu_j \right\|^2_{\ell^2} 
	& = \left\| (\rI - {\rU}{\rU}^T)\times_1{\bPhi} \right\|^2_F 
	= \left\| (\rI - {\rU}{\rU}^T)\times_1\left( {\bPhi} - \widetilde{\bPhi} \right) \right\|^2_F \\
	& \le \left\| \rI - {\rU}{\rU}^T \right\|^2 
	\left\| {\bPhi} - \widetilde{\bPhi}  \right\|^2_F \le 	\left\| {\bPhi} - \widetilde{\bPhi}  \right\|^2_F 
	 ~\le~ {\eps}^2\big\|\bPhi \big\|_F^2.
\end{split}
\end{equation}
}

\noindent where we used  
$\| \rA\times_k \mathbf{B} \|_F = \| \rA\rB_{(k)} \|_F \le \| \rA \| \|  \rB_{(k)} \|_F=\| \rA \| \|  \mathbf{B} \|_F$ 
for a matrix $\rA$ and a tensor $\mathbf{B}$ of compatible sizes, and spectral matrix norm $\| \cdot \|$. 
We also used $\| \rP \| = 1$ for the orthogonal projection matrix $\rP = \rI - {\rU} {\rU}^T$. 
 
The bound  \eqref{eqn:tphibound} and a similar bound for $\blf$-snapshots resembles  the POD optimal 
representation property \eqref{eqn:phipsibound}. Universal spaces $\widetilde{U}$ and $\widetilde{Y}$ 
can be seen as TROM counterparts of POD--DEIM ROM spaces  $U_{\rm pod}$ and $Y_{\rm pod}$ 
In fact, if SVD-based algorithms from~\cite{de2000multilinear,TT1} are applied to find $\bPsi$ and $\bPhi$ 
in Tucker or TT-formats, then it  holds $\widetilde{U}=U_{\rm pod}$ and $\widetilde{Y}=Y_{\rm pod}$ if 
$\widetilde N=n_{\rm pod}$ and the same training set is used for both POD and TROM. 
The advantage of LRTD over POD is that $\widetilde{\bPhi}$ and $\widetilde{\bPsi}$ contain information about variation 
of $\bu$- and $\blf$-snapshots with respect to $\balpha$. This additional information enables us to find the 
subspaces of $\widetilde{U}$ and $\widetilde{Y}$ referred to as local reduced spaces that are best suitable 
for the representation of $\bu(t,\balpha)$ and $\blf_{\balpha}(t,\bu(t,\balpha))$ for any specific $\balpha\in\cA$.
These local reduced spaces are then used for online TROM's stage. 
Their dimensions can be (much) lower than the dimensions of $\widetilde{U}$ and $\widetilde{Y}$, 
thus the universal spaces can be allowed to be sufficiently large (by choosing $\eps$ small enough) 
to accurately represent all $\bu$- and $\blf$-snapshots without any  decrease of ROM performance
at the online stage.

To define parameter-specific local reduced spaces for an arbitrary $\balpha =(\alpha_1,\dots,\alpha_D)^T\in \cA$, 
we need an interpolation procedure in parameter domain
\begin{equation}
	\label{eqn:bea}
	\be^i \,:\, \balpha \to \mathbb{R}^{K_i},\quad i=1,\dots,D.
\end{equation}
such that for a smooth function 
${g} : [\alpha_i^{\min}, \alpha_i^{\max}] \to \R$ one approximates
	${g}(\alpha_i) \approx 	\sum\limits_{j=1}^{K_i} e_{j}^i (\balpha) {g}(\widehat{\alpha}_i^j)$,
where $\be^i(\balpha) = \big(e_{ 1}^i (\balpha), \ldots, e_{ K_i}^i (\balpha) \big)^T$, and  $\widehat{\alpha}_i^j$, $j=1,\ldots,K_i$, are the grid nodes on $[\alpha_i^{\min}, \alpha_i^{\max}]$.
%
In this paper we consider \eqref{eqn:bea} corresponding to Lagrange 
interpolation of {order $p-1$}: for a given $\balpha \in \cA$ let 
$\widehat{\alpha}_i^{i_1}, \ldots, \widehat{\alpha}_i^{i_p}$ be the $p$ closest grid nodes to $\alpha_i$ 
in $[\alpha_i^{\min}, \alpha_i^{\max}]$, for $i=1,\ldots,D$, then
\begin{equation}
	\label{eqn:lagrange}
	e_{j}^i (\balpha) = 
	\begin{cases} 
		\prod\limits_{\substack{m = 1, \\ m \neq k}}^{p}(\widehat{\alpha}_i^{i_m}-\alpha_i) \Big/ 
		\prod\limits_{\substack{m = 1, \\ m \neq k}}^{p}(\widehat{\alpha}_i^{i_m}-\widehat{\alpha}_i^j), 
		& \text{if } j = i_k \in \{i_1,\ldots,i_p\}, \\
		\qquad\qquad\qquad\qquad\qquad\qquad\qquad 0, & \text{otherwise}, \end{cases}
\end{equation}
Lagrange interpolation is not the only possible option, of course. 

With the help of \eqref{eqn:bea} we introduce the local snapshot matrices 
$\widetilde{\Phi} (\balpha)$ and $\widetilde{\Psi} (\balpha)$ via the following ``extraction--interpolation'' procedure:
\begin{equation}
	\label{eqn:extractbt}
	\begin{split}
	\widetilde{\Phi} (\balpha) &= \widetilde{\bPhi} 
	\times_2 \be^1(\balpha) \times_3 \be^2(\balpha) \dots \times_{D+1} \be^D(\balpha) 
	\in \R^{M \times N},\\
	\widetilde{\Psi} (\balpha) &= \widetilde{\bPsi} 
\times_2 \be^1(\balpha) \times_3 \be^2(\balpha) \dots \times_{D+1} \be^D(\balpha) 
\in \R^{M \times N}.
	\end{split}
\end{equation}
If $\balpha = \bhalpha \in \hcA$ 
is a parameter from the training set, then $\be^i(\hbalpha)$ encodes the position of $\widehat{\alpha}_i$ 
among the grid nodes on $[\alpha^{\min}_i, \alpha^{\max}_i]$. Therefore, for $\eps=0$ the  
matrices $\widetilde{\Phi} (\hbalpha)$, $\widetilde{\Psi} (\hbalpha)$ are exactly the matrices of all  
$\bu$- and $\blf$-snapshots for the particular $\hbalpha$ (``extraction''). Otherwise, for a general $\balpha\in \cA$  matrices
$\widetilde{\Phi} (\balpha)$, $\widetilde{\Psi} (\balpha)$ are the result of interpolation between pre-computed snapshots.

Finally, we have all the required pieces to define the local  spaces. For arbitrary given $\balpha \in \cA$ 
the parameter-specific \emph{local reduced $\bu$-space of dimension $n$ is the space spanned by the first 
$n$ left singular vectors of $\widetilde{\Phi} (\balpha)$}, where $ n \le \mbox{rank}(\widetilde{\Phi} (\balpha)) $.
Similarly, the parameter-specific local reduced $\blf$-space of dimension $n$ is  spanned by 
the first $n$ left singular vectors of $\widetilde{\Psi} (\balpha)$,
$n \le \mbox{rank}(\widetilde{\Psi} (\balpha))$. It is quite remarkable that orthogonal bases for each of 
these local spaces can be calculated quickly (i.e., using only low-dimensional calculations) through their 
coordinates in the corresponding universal spaces \emph{without assembling} 
$\widetilde{\Phi} (\balpha)$ or $\widetilde{\Psi} (\balpha)$ explicitly. TROM framework performs this ``on the fly'' during the online stage for any incoming $\balpha$,
as explained  later.

\smallskip
In summary, the universal spaces, within a compression accuracy of $\epsilon$, encompass the spaces 
\begin{wrapfigure}{r}{0.45\textwidth}
	\vskip-1ex 		
	\includegraphics[width=0.45\textwidth]{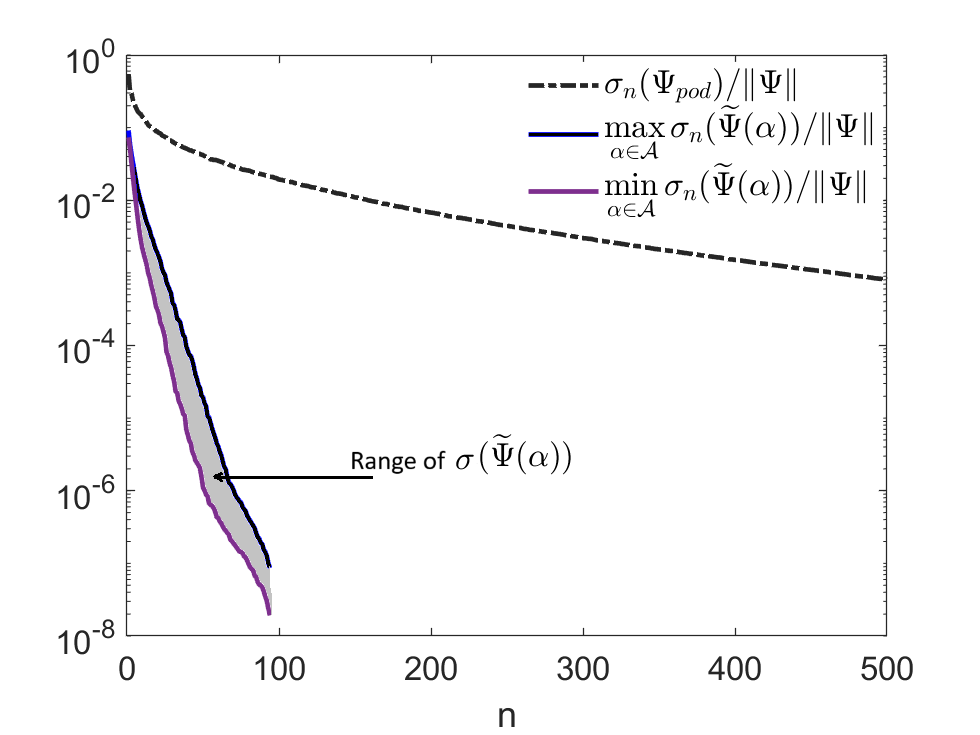}
	\vskip-2ex 		
	\caption{\label{fig1} Singular values $\sigma_n({\Psi}_{\rm pod})$ and $\sigma_n(\widetilde{\Psi}(\balpha))$ 
		for an example of the parametrized Allen-Cahn equations from Section~\ref{sec:AC}.}
\end{wrapfigure}
formed by all observed snapshots. These spaces resemble those utilized by the conventional POD-DEIM ROM, where the SVD of matrices containing all $\bu$- and $\blf$-snapshots is truncated based on the desired dimension or approximation accuracy. In the TROM framework, this truncated SVD is substituted with LRTD in one of the formats \eqref{eqn:CPv}--\eqref{eqn:TTv}.
Beyond the universal reduced spaces, the two-stage TROM framework also takes advantage of $\balpha$-specific 'subspaces' within the universal spaces, known as the local reduced spaces.
We illustrate the usefulness of local reduced spaces through a numerical example of the parametrized Allen-Cahn equations (for a complete problem description, see Section~\ref{s:num}) in Figure~\ref{fig1}. This figure compares the singular values of $\Psi_{\text{pod}} = \Psi_{(1)}$ with the singular values of $\widetilde{\Psi}(\balpha)$ for a random sampling of $\balpha \in \cA$. Truncation by $\sigma_n({\Psi}_{\text{pod}})$ and $\sigma_n(\widetilde{\Psi}(\balpha))$ determines the representation power of the standard POD-DEIM ROM interpolation space and the TROM local space of a fixed dimension $n$. We observe a significant enhancement in the representation power of TROM local spaces compared to POD-DEIM ROM spaces, demonstrating the benefits of taking into account information on $\balpha$-dependence during the offline decomposition stage.

\subsection{Two-stage TROM-DEIM process} 

With all the preliminaries in order, we are now ready to introduce the two-stage process for computing
tensorial reduced order model that incorporates DEIM to handle nonlinearity. 
One may formulate at least three variants of this process, each corresponding to the considered LRTD format. However, the main steps  remain the same for all three
variants. Therefore, we will present the TT version of the TROM-DEIM process, with additional comments regarding the necessary adjustments for other LRTD formats provided at the end of this section.

\textbf{Offline stage}. The first calculation is to compute TT decomposition to both
$\bu$-snapshot tensor $\bPhi$ and $\blf$-snapshot tensor $\bPsi$, satisfying \eqref{eqn:TensApprox}:
\begin{equation}\label{TTl}
\begin{split}
	\bPhi\approx \widetilde{\bPhi} &=
		\mbox{\small $\displaystyle \sum_{j_1=1}^{\widetilde R_1^\Phi}\dots
	\sum_{j_{D+1}=1}^{\widetilde R^\Phi_{D+1}}$}
	\bu^{j_1} \circ \bsigma^{j_1, j_2}_{1,\Phi} \circ \dots \circ \bsigma^{j_D, j_{D+1}}_{D,\Phi} \circ \bv^{j_{D+1}},\\
	\bPsi\approx \widetilde{\bPsi} &=
		\mbox{\small $\displaystyle \sum_{j_1=1}^{\widetilde R_1^\Psi}\dots
	\sum_{j_{D+1}=1}^{\widetilde R_{D+1}^\Psi}$}
	\by^{j_1} \circ \bsigma^{j_1, j_2}_{1,\Psi} \circ \dots \circ \bsigma^{j_D, j_{D+1}}_{D,\Psi} \circ \bz^{j_{D+1}},
\end{split}
\end{equation}
with $\bu^{j_1},\by^{j_1} \in \R^M$, $\bsigma^{j_i, j_{i+1}}_{i,\Phi},\bsigma^{j_i, j_{i+1}}_{i,\Psi} \in \R^{K_i}$, and $\bv^{j_{D+1}},\bz^{j_{D+1}} \in \R^N$.
TT decompositions $\widetilde{\bPhi}$, $\widetilde{\bPsi}$ may be found by using a stable algorithm  based on truncated SVD for a sequence of unfolding matrices~\cite{TT1,TT2} with computational costs similar to finding 
${\Phi}_{\rm pod}$ and ${\Psi}_{\rm pod}$ for the conventional POD--DEIM ROM.
 We organize  vectors from \eqref{TTl} into matrices 
\begin{align}
\label{eqn:ttuv}
\rU & = [\bu^1, \dots, \bu^{\widetilde{R}_1^\Phi}] \in \R^{M \times \widetilde{R}_1^\Phi}, \quad 
\rV = [\bv^1, \dots, \bv^{\widetilde{R}_{D+1}^\Phi}] \in \R^{N \times \widetilde{R}_{D+1}^\Phi}, \\
\label{eqn:ttyz}
\rY & = [\by^1, \dots, \by^{\widetilde{R}_1^\Psi}] \in \R^{M \times \widetilde{R}_1^\Psi}, \quad 
\rZ = [\bz^1, \dots, \bz^{\widetilde{R}_{D+1}^\Psi}] \in \R^{N \times \widetilde{R}_{D+1}^\Psi},
\end{align}
where $\widetilde{R}_1^\Theta$ and $\widetilde{R}_{D+1}^\Theta$ are the first and
last TT ranks of both snapshot tensors $\tbTheta \in \{ \tbPhi, \tbPsi \}$, respectively.
We also consider third order tensors 
$\bS_i^{\Theta} \in \R^{\widetilde{R}_i^\Theta \times K_i \times \widetilde{R}_{i+1}^\Theta}$, 
defined entry-wise as
\begin{equation}
	\left[ \bS_i^{\Theta} \right]_{j k q} = \left[ \bsigma^{j q}_{i,\Theta} \right]_k, \quad
	j = 1, \ldots, \widetilde{R}_i^\Theta, \quad
	k = 1, \ldots, K_i, \quad
	q = 1, \ldots, \widetilde{R}_{i+1}^\Theta,
	\label{eqn:ttsi}
\end{equation}
for all $i = 1, \ldots, D$, for both snapshot tensors $\tbTheta \in \{ \tbPhi, \tbPsi \}$, respectively.
While both $\rU$ and $\rY$ are orthogonal matrices, the columns of $\rV$ and $\rZ$ are orthogonal, 
but not necessarily orthonormal. Thus, we need diagonal scaling matrices
\[
\rW^{\Phi}  = \mbox{diag} \left( \| \bv^1 \|, \ldots, \| \bv^{\widetilde{R}_{D+1}^\Phi} \| \right), \quad 
\rW^{\Psi}  = \mbox{diag} \left( \| \bz^1 \|, \ldots, \| \bz^{\widetilde{R}_{D+1}^\Psi} \| \right) .
\]
Then, the essential information about the compressed TT representations 
$\widetilde{\bTheta} \in \{ \widetilde{\bPhi},  \widetilde{\bPsi} \}$ that is transmitted to the online phase 
is assembled into
\begin{equation}
\label{eqn:coreTT}
\mbox{core}(\tbTheta) = 
\left\{\bS_i^{\Theta} \in \R^{\widetilde{R}_i^\Theta \times K_i \times \widetilde{R}_{i+1}^\Theta}
,~{\small i=1,\dots,D},~ 
\rW^{\Theta} \in \R^{\widetilde{R}_{D+1}^\Theta \times \widetilde{R}_{D+1}^\Theta} \right\}.
\end{equation}

To perform hyper-reduction, DEIM algorithm is applied to the orthonormal columns 
of $\rY$ to compute the indices $\bseta$ of the selection matrix 
$ \rP  = [\be_{\eta_1}, \ldots, \be_{\eta_{\widetilde{R}_1^\Psi}}] 
\in \mathbb{R}^{M \times \widetilde{R}_1^\Psi}$. 
Then, the matrices
\begin{equation}
\rU^T \rY \in \mathbb{R}^{\widetilde{R}_1^\Phi \times \widetilde{R}_1^\Psi} \quad\mbox{and}\quad 
 \rP ^T \rY \in \mathbb{R}^{\widetilde{R}_1^\Psi \times \widetilde{R}_1^\Psi}
\label{eqn:utypty}
\end{equation}
are computed and passed onto the online stage along with the TT cores \eqref{eqn:coreTT}.

If one or both terms $\rA_{\balpha}$ and $\blf_{\balpha}$ in \eqref{eqn:GenericPDE} do not contain
a dependence on the parameter, they can be projected for later use to save computation at 
the online stage:
\[
\widetilde{\rA}  = \rU^T \rA \rU, \quad
\widetilde{\blf}(\cdot)  =  \rP ^T \blf(\rU \cdot).
\]
Analogously, the projections can be pre-computed offline if the dependence on parameters is explicit of the form
$\rA_{\balpha}=\sum_i g_i(\balpha)\rA_i$ with some given functions $g_i:\cA\to\mathbb{R}$ and parameter-free matrices $\rA_i$ and similarly for $\blf_{\balpha}$.

\smallskip
\textbf{Online stage}. The second stage of the TROM process is specific for a particular
incoming value of $\balpha \in \cA$. In particular, it computes the two local reduced bases via 
their coordinates in the universal reduced bases, the columns of $\rU$
and $\rY$, respectively. To achieve this, use the cores \eqref{eqn:coreTT} to define the 
parameter-specific \emph{core matrices} 
$\rC^{\Theta} (\balpha) \in \R^{\widetilde{R}_1^\Theta \times \widetilde{R}_{D+1}^\Theta}$ 
as the product
\begin{equation}
	\label{eqn:C_TT}
	\rC^{\Theta} (\balpha) = \mbox{\small $\displaystyle \prod_{i=1}^{D}$} \left( \bS_i^{\Theta} \times_2 \be^i (\balpha) \right),
\end{equation}
for $\tbTheta \in \{ \tbPhi, \tbPsi \}$. After rescaling with $\rW^{\Theta}$, take the SVD of the rescaled 
core matrices 
\begin{equation}
\rC^{\Phi} (\balpha) \rW^{\Phi} = \rU_c \Sigma_c^{\Phi} \rV_c^T, \quad
\rC^{\Psi} (\balpha) \rW^{\Psi} = \rY_c \Sigma_c^{\Psi} \rZ_c^T,
\label{eqn:coresvdtt}
\end{equation}
which is computationally cheap since $\rC$'s and $\rW$'s have reduced dimensions. 
This allows one to obtain the SVD of the local snapshot matrices from \eqref{eqn:extractbt} without explicitly assembling them. To see this, note the identities 
\begin{align}
\label{eqn:PhieTT}
\widetilde{\Phi} (\balpha) & = 
\rU \rC^{\Phi} (\balpha) \rW^{\Phi} \left( \rW^{\Phi} \right)^{-1} \rV^T = 
\left({\rU} \rU_c \right) \Sigma_c^{\Phi} \left({\rV} \left( \rW^{\Phi} \right)^{-1} \rV_c \right)^T, 
\end{align}
and similar for $\widetilde{\Psi} (\balpha)$.
Since all matrices $\rU$, $\rU_c$, ${\rV} ( \rW^{\Phi})^{-1}$, $\rV_c$,  have orthonormal columns,  the right-hand sides of \eqref{eqn:PhieTT} are the (thin) SVDs of 
$\widetilde{\Phi} (\balpha)$ and $\widetilde{\Psi} (\balpha)$, respectively.
Therefore, the coordinates of the local reduced $\bu$-basis in the universal space $\widetilde{U}$ 
are given by the first $n$ columns of $\rU_c$. Similarly, the coordinates of the local reduced $\blf$-basis 
in the universal space $\widetilde{Y}$ are the first $n$ columns of $\rY_c$. We denote $\rU_n = [\rU_c]_{:, 1:n} \in \mathbb{R}^{\widetilde{R}_1^\Phi \times n}$ and 
$\rY_n = [\rY_c]_{:, 1:n} \in \mathbb{R}^{\widetilde{R}_1^\Psi \times n}$. Thus the local projection and interpolation spaces are $\mbox{range}(\rU\rU_n)$ and $\mbox{range}(\rY\rY_n)$.

After computing the coordinates of local reduced bases, an additional hyper-reduction step is performed.
Below we introduce two variants of this \emph{local} hyper-reduction procedure: (i)  local DEIM and (ii) local least squares fitting. 

\smallskip
\noindent{\it (i) Local DEIM} is done by applying DEIM algorithm to $\left( \rP^T \rY \right) \rY_n$ to obtain the indices  $\mathbf{\xi} = [\xi_1,\ldots,\xi_n]$ and the corresponding selection matrix
$\rP_n=\rP_n(\mathbf{\xi}) \in \mathbb{R}^{\widetilde{R}_1^\Psi \times n}$. 
Then, the non-linear term can be found using the matrices \eqref{eqn:utypty} precomputed
at the offline stage:
\begin{equation}
\blf_n (\cdot) = \rU_n^T \left( \rU^T \rY \right) \rY_n 
\left( \rP^T_{n} \left( \rP^T \rY \right) \rY_n \right)^{-1} 
\rP^T_{n} \widetilde{\blf}(\rU_n \cdot),
\label{eqn:fn}
\end{equation}

\noindent{\it (ii) Local LS fitting.}  Thanks to reduced dimensions of the universal and local space, it is computationally  inexpensive to solve the fitting problem of finding $\bx\in\R^n$ such that $\left( \rP^T \rY \right) \rY_n \bx \approx \widetilde{\blf}(\rU_n \cdot)$  in the least square sense. 
Then, the non-linear term representation takes the form
\begin{equation}
	\blf_n (\cdot) = \rU_n^T \left( \rU^T \rY \right) \rY_n 
	\left(  \left( \rP^T \rY \right) \rY_n \right)^{\dagger} \widetilde{\blf}(\rU_n \cdot),
	\label{eqn:fn_ls}
\end{equation}
where $((\rP^T \rY ) \rY_n )^{\dagger}$ denotes the pseudo-inverse of 
 $\widetilde{R}_1^\Psi \times n$-matrix  $( \rP^T \rY ) \rY_n$.

Along with either variant of local hyper-reduction,  the pre-projected matrix $\widetilde{A}$ is projected further to obtain
\begin{equation} \label{eqA}
\rA_n = \rU_n^T \widetilde{\rA} \rU_n.
\end{equation}

Finally, the TROM of \eqref{eqn:GenericPDE} takes the form: Find $\bbeta(t):[0,T]\to \mathbb{R}^n$ solving
\begin{equation} \label{ROMeq}
\bbeta_t = \rA_n \bbeta + \blf_n (\bbeta),  \quad t \in (0,T), 
\end{equation}
with the initial condition
$\bbeta|_{t=0} = \rU_n^T \rU^T \bu_0$,
so that $\bu(t)$ is approximated by 
$\bu^{\rm trom}(t):[0,T] \to \R^M$, 
where $\bu^{\rm trom}(t)=\rU \rU_n \bbeta(t)$. 

We summarize both stages described above in the Algorithm~\ref{Alg1}. In practice, one may pick 
different reduced dimensions $n$ for the local reduced $\bu$- and $\blf$-spaces. When we need to distinguish 
them, we use $n_{\Phi}$ and $n_{\Psi}$, respectively.

\begin{algorithm}\caption{TROM-DEIM} \label{Alg1}
	\label{alg:TROM-DEIM}
	\begin{itemize}
		\item \textbf{Offline stage}.\\
		\emph{Input:} snapshot tensors $\bPhi \in \R^{M \times K_1 \times \ldots \times K_D \times N}$ and
		$\bPsi \in \R^{M \times K_1 \times \ldots \times K_D \times N}$ and
		target accuracy $\eps$;\\
		\emph{Output:} Compression ranks, decomposition cores \eqref{eqn:coreTT}, 
		matrices \eqref{eqn:utypty}; \\
		\emph{Compute:} 
		\begin{enumerate}
			\item Use algorithm from~\cite{TT1} with prescribed accuracy $\eps$ to compute  
			decomposition \eqref{eqn:TTv} for both $\bTheta \in \{ \bPhi, \bPsi\}$; assemble 
			the cores \eqref{eqn:coreTT} and the matrices $\rU$ and $\rY$ as in 
			\eqref{eqn:ttuv}--\eqref{eqn:ttyz};
			\item Apply DEIM to $\rY$ to find the selection matrix $\rP$; 
			\item Compute the matrices $\rU^T \rY$ and $ \rP ^T \rY$;
	\end{enumerate}
		\item \textbf{Online stage}. \\
		\emph{Input:} decomposition cores \eqref{eqn:coreTT}, local reduced space dimensions 
		\begin{equation} 
		n_{\Phi} \le \min\{\widetilde{R}_1^\Phi, \widetilde{R}_{D+1}^\Phi\},\quad
		n_{\Psi} \le \min\{  \widetilde{R}_1^\Psi, \widetilde{R}_{D+1}^\Psi \},                
		\label{eqn:nphipsi}
		\end{equation} 
		and an incoming parameter vector $\balpha \in \cA$; \\
		\emph{Output:} Coordinates of the local reduced bases in the form of matrices 
		$\rU_{n}$ and $\rY_{n}$;\\
		\emph{Compute:}
		\begin{enumerate}
			\item Use tensors $\bS_i^{\Theta}$, $i=1,\ldots,D$, to assemble the core matrices
			$\rC^{\Theta} (\balpha) \in \R^{\widetilde{R}_1^\Theta \times \widetilde{R}_{D+1}^\Theta}$ 
			as in \eqref{eqn:C_TT} for both $\Theta \in \{ \Phi, \Psi\}$;
			\item Compute the SVD \eqref{eqn:coresvdtt} of both scaled core matrices 
			$\rC^{\Phi} (\balpha) \rW^{\Phi}$ and $\rC^{\Psi} (\balpha) \rW^{\Psi}$ 
			to find the matrices of left singular vectors $\rU_c$ and $\rY_c$, respectively;
			\item Set $\rU_n = [\rU_c]_{:, 1:n_{\Phi}}$ and $\rY_n = [\rY_c]_{:, 1:n_{\Psi}}$.
			
		\end{enumerate}
	\end{itemize}
\end{algorithm}

After the model reduction is done according to Algorithm~\ref{Alg1}, the projected problem 
\eqref{ROMeq} is solved with either \eqref{eqn:fn}, \eqref{eqA} or 
\eqref{eqn:fn_ls}, \eqref{eqA}. One lets $\bu^{\rm trom}(t)=\rU \rU_n \bbeta(t)$. 
For another incoming vector of parameters $\balpha$, the online part of Algorithm~\ref{Alg1} 
should be recomputed. 

\begin{remark}[Computational complexity]Let us introduce 
$\widetilde{R} = 
\max\{ \widetilde{R}_1^\Phi, \widetilde{R}_{D+1}^\Phi,  \widetilde{R}_1^\Psi, \widetilde{R}_{D+1}^\Psi \}$
and $n = \max\{ n_\Phi, n_\Psi \}$. Then, the bulk of the computational cost of the online stage of TROM-DEIM is in the 
SVD computation \eqref{eqn:coresvdtt} which takes $O(\widetilde{R}^3)$. This is followed by integration of
\eqref{ROMeq} costing $O(n^2 N)$. This compares favorably to POD-ROM integration cost of  $O(\widetilde{R}^2 N)$,
provided $\widetilde{R} < N$ which is always the case for reasonably long integration times. Then, the speed up of 
TROM--DEIM relative to POD--ROM is determined by the ratio of dimensions of local and universal spaces.
\end{remark}

We conclude with a brief discussion of variants of Algorithm~\ref{Alg1}
based on the other two LRTD formats by pointing out the required modifications. 

\textbf{HOSVD-TROM}. For Step 1 of the offline stage we employ HOSVD-LRTD \eqref{eqn:TDv}
instead of TT-LRTD \eqref{eqn:TTv}, so the cores \eqref{eqn:coreTT} are replaced with
\begin{equation}
\label{eqn:coreHOSVD}
\mbox{core}(\widetilde{\bTheta}) = \left\{
\bC^\Theta \in 
\R^{\widetilde{M}^\Theta \times \widetilde{K}_1^\Theta \times \dots \times \widetilde{K}_D^\Theta 
\times \widetilde{N}^\Theta},~
\rS_i^\Theta \in \mathbb{R}^{n_i \times \widetilde{K}_i^\Theta},~{\small i=1,\dots,D} \right\},
\quad \widetilde{\bTheta} \in \{ \widetilde{\bPhi},  \widetilde{\bPsi} \},
\end{equation}
where
$\rS_i^\Theta = \left[ \bsigma^1_{i, \Theta}, \dots, \bsigma^{\widetilde{K}_i^\Theta}_{i, \Theta} \right]^T \in 
\mathbb{R}^{{n}_i \times \widetilde{K}_i^\Theta}$, $i = 1,\ldots,D$.
The matrices $\rU$ and $\rY$ are still assembled as in \eqref{eqn:ttuv}--\eqref{eqn:ttyz}, but
the vectors $\bu^i$, $i=1,\ldots,\widetilde{M}^\Phi$, and $\by^j$, $j=1,\ldots,\widetilde{M}^\Psi$
come from HOSVD-LRTD \eqref{eqn:TDv} instead.
At the online stage of Algorithm~\ref{Alg1} the bound \eqref{eqn:nphipsi} is replaced by
$n_{\Phi} \le \min\{\widetilde{M}^\Phi, \widetilde{N}^\Phi\}$ ,
$n_{\Psi} \le \min\{  \widetilde{M}^\Psi, \widetilde{N}^\Psi \}$,
while the core matrices \eqref{eqn:C_TT} at Step 1 are instead computed as
\begin{equation}
\label{eqn:C_hosvd}
\rC^\Theta (\balpha) = \bC^\Theta \times_2 
\left(\rS_1^\Theta \be^1 (\balpha) \right) \times_3 \left(\rS_2^\Theta \be^2 (\balpha)\right) 
\dots\times_{D+1} \left(\rS_D^\Theta \be^D (\balpha)\right)  
\in \R^{\widetilde{M}^\Theta \times \widetilde{N}^\Theta}, \quad
\Theta \in \{ \Phi, \Psi \}.
\end{equation}
Unlike the TT case, no rescaling is needed for HOSVD core matrices \eqref{eqn:C_hosvd}, so the
SVD in Step~2 of the online stage becomes simply
\begin{equation} 
\label{eqn:coresvd}
\rC^{\Phi} (\balpha) = \rU_c \Sigma_c^{\Phi} \rV_c^T, \quad
\rC^{\Psi} (\balpha) = \rY_c \Sigma_c^{\Psi} \rZ_c^T.
\end{equation}

\textbf{CP-TROM}. 
For CP-TROM, the replacement of LRTD in Step 1 of the offline stage with \eqref{eqn:CPv} can be computed through two distinct approaches. Firstly, one can employ a Proper Generalized Decomposition method to compute CP-LRTD incrementally in a greedy manner, gradually increasing the CP-ranks $\widetilde{R}^\Phi$ and $\widetilde{R}^\Psi$ until the target accuracy $\epsilon$ is achieved \cite{diez2018algebraic}. Alternatively, the target ranks $\widetilde{R}^\Phi$ and $\widetilde{R}^\Psi$ can be specified initially, and an ALS algorithm can be used to compute CP-LRTD \cite{ReviewTensor}. In this context, we opt for the latter approach due to the availability of high-performance ALS implementations.

For the matrices
\begin{align}
\label{eqn:cpuv}
\rU & = [\bu^1, \dots, \bu^{\widetilde{R}^\Phi}] \in \R^{M \times \widetilde{R}^\Phi}, \quad 
\rV = [\bv^1, \dots, \bv^{\widetilde{R}^\Phi}] \in \R^{N \times \widetilde{R}^\Phi}, \\
\label{eqn:cpyz}
\rY & = [\by^1, \dots, \by^{\widetilde{R}^\Psi}] \in \R^{M \times \widetilde{R}^\Psi}, \quad 
\rZ = [\bz^1, \dots, \bz^{\widetilde{R}^\Psi}] \in \R^{N \times \widetilde{R}^\Psi},
\end{align}
their thin QR factorizations are computed
$\rU = \rQ_U \rR_U$, $\rV = \rQ_V \rR_V$,  
$\rY = \rQ_Y \rR_Y$, $\rZ = \rQ_Z \rR_Z$,
in order to obtain the cores
\begin{eqnarray}
\mbox{core}(\widetilde{\bPhi}) & = 
\{ \rR_U, \rR_V \in \mathbb{R}^{\widetilde{R}^\Phi \times \widetilde{R}^\Phi}, 
\bsigma_{i, \Phi}^r, \; i=1,\ldots,D,\; r=1,\ldots, \widetilde{R}^\Phi \}, \\ 
\mbox{core}(\widetilde{\bPsi}) & = 
\{ \rR_Y, \rR_Z \in \mathbb{R}^{\widetilde{R}^\Psi \times \widetilde{R}^\Psi}, 
\bsigma_{i, \Psi}^r, \; i=1,\ldots,D,\; r=1,\ldots, \widetilde{R}^\Psi \}.
\end{eqnarray}

At the online stage instead of \eqref{eqn:nphipsi} we simply require that $n_\Theta \leq \widetilde{R}^\Theta$
for both $\Theta \in \{ \Phi, \Psi \}$. At Step 1, the core matrices \eqref{eqn:C_hosvd} are replaced with
\begin{equation}
\rC^\Phi (\balpha) = \rR_U \rS^\Phi (\balpha) \rR_V^T 
\in \mathbb{R}^{\widetilde{R}^\Phi \times \widetilde{R}^\Phi}, \quad
\rC^\Psi (\balpha) = \rR_Y \rS^\Psi (\balpha) \rR_Z^T
\in \mathbb{R}^{\widetilde{R}^\Psi \times \widetilde{R}^\Psi},
\end{equation}
where
\begin{equation}
\rS^\Phi (\balpha) = \mbox{diag}(s_1^\Theta,\ldots, s_{\widetilde{R}^\Theta}^\Theta), \mbox{~with~~}
s_r^\Theta = \prod_{i=1}^{D} \left\langle \bsigma_{i, \Theta}^r, \be^i(\balpha) \right\rangle, \quad 
r = 1,\ldots, \widetilde{R}^\Theta,
\end{equation}
for both $\Theta \in \{ \Phi, \Psi \}$. Similarly to the HOSVD variant, no rescaling of the core matrices is 
needed, so the SVD in Step 2 of the offline stage is \eqref{eqn:coresvd}.

\section{Representation and interpolation estimates} 
\label{s:analysis}

In this section we consider representation capacity of TROM local reduced bases and prove an interpolation 
estimate for the two-stage TROM--DEIM. To do so, we need a few assumptions about the properties of 
the dynamical system \eqref{eqn:GenericPDE}. In particular, we assume that the unique solution to
\eqref{eqn:GenericPDE} exists on $(0,T_1)$ for all $\balpha \in\cA_1$, with some $T_1 > T$ 
and $\overline{\cA} \subset \cA_1$. Also, $\blf_{\balpha}$ is continuous with continuous derivatives in $\bu$ and $\balpha$ up to order $p$: 
 \begin{equation}
 	\label{eqn:Assf}
 	\blf_{\balpha}=\blf(t,\bu,\balpha) \in {C([0,T_1],C^p(\mathbb{R}^M\times\cA_1))},
 \end{equation}
where $p$ is interpolation order parameter from \eqref{eqn:lagrange}. 
For a vector function we define its $C^p$ norm as maximum over all components 
$\|\blf\|_{C^p}= \max_i \|f_i\|_{C^p}$.

The assumption in \eqref{eqn:Assf} implies that  the solution of \eqref{eqn:GenericPDE} is  smooth with 
respect to parameters. More precisely,  it holds (cf. \cite[Theorem~V.4.1]{Hartman}):
\begin{equation}
	\label{eqn:AssU}
	\bu \in C([0,T] \times \overline{\cA})^M, \quad 
	\frac{\partial^{\mathbf{j}} \bu}{\partial \alpha^{j_1}_1\dots\partial\alpha^{j_D}_D} 
	\in C([0,T] \times \overline{\cA})^M, \quad |\mathbf{j}| \le p.
\end{equation}
Letting
$
C_\bu= \max\limits_{|\mathbf{j}| \le p} \left\| 
\frac{\partial^{\mathbf{j}} \bu}{\partial \alpha^{j_1}_1\dots\partial\alpha^{j_D}_D} 
\right\|_{C([0,T] \times \overline{\cA})},
$
we apply the chain rule and use \eqref{eqn:AssU} to estimate
\begin{equation}
\label{eqn:setF}
\max\limits_{|\mathbf{j}| \le p} \left\|
\frac{\partial^{\mathbf{j}} \blf(t,\bu(\balpha),\balpha)}
{\partial \alpha^{j_1}_1\dots\partial\alpha^{j_D}_D}
\right\|_{C([0,T] \times \overline{\cA})}\le C\, {\sup_{t\in(0,T)}\|\blf(t)\|_{C^p}}(1+C_\bu^p)=: C_{\blf},
\end{equation}
with some $C_{\blf}$ independent of FOM dimensions.

For an arbitrary fixed $\balpha = (\alpha_1,\ldots,\alpha_D)^T \in \cA$, 
not necessarily from the sampling set, consider the FOM $\bu$-snapshots of \eqref{eqn:GenericPDE} 
and the corresponding $\blf$-snapshots,
 \[
\bu^k=\bu(t_k,\balpha),\quad \blf^k = \blf_{\balpha}(t_k,\bu^k),
\quad \text{for}~~0\le t_1,\dots,t_N<T.
 \]

Denote by $\widehat\bu_i$, $i=1,\dots,n$, the basis vectors of the local reduced $\bu$-space,  i.e., 
$\widehat\bu_i$ are the first $n$ left singular vectors of $\widetilde\Phi (\balpha)$. This local basis used 
to represent TROM solution admits the following representation estimate~\cite{mamonov2022interpolatory}: 
\begin{equation}
	\label{eqn:est1}
	\frac1{NM}\mbox{\small $\displaystyle \sum_{k=1}^{N}$} 
	\left\| \bu^k - \mbox{\small $\displaystyle\sum_{j=1}^{n} $}
	\left\langle \bu^k, \widehat\bu_j \right\rangle \widehat\bu_j \right\|^2_{\ell^2} \\
	\le \frac{C_1}{NM} \left(\eps^2\left\| \bPhi \right\|_F + 
	\mbox{\small $\displaystyle\sum_{i = n+1}^{N}$} \tsigma_i^2 \right) + C_2 \delta^{2p} ,
\end{equation}
where $\tsigma_i$ are the singular values of $\widetilde{\Phi}(\balpha)$ and 
$\delta_i$ are the mesh parameter of the sampling,
\begin{equation}
	\delta_i = \max\limits_{1 \leq j  \leq K_i-1} \left| \widehat{\alpha}_j^i - \widehat{\alpha}_{j+1}^i \right|, 
	\quad i = 1,\ldots,D,\quad\text{and}~\delta^p=\textstyle\sum_{i=1}\limits^D\delta_i^p.
\end{equation}
Constants $C_1$ and $C_2$ in \eqref{eqn:est1} depend only on the stability of the interpolation 
procedure and bounds on partial derivatives of $\bu$ from \eqref{eqn:AssU}.  The scaling $1/(NM)$ 
accounts for the variation of dimensions $N$ and $M$, which may correspond to the number of 
temporal and spatial degrees of freedom, if \eqref{eqn:GenericPDE} results from a discretization of a 
parabolic PDE. 
In this case and for uniform grids, the quantity on the left-hand side of \eqref{eqn:est1} is consistent 
with the $L^2(0, T, L^2(\Omega))$ norm. 

We now want to derive an interpolation bound for the two-stage TROM--DEIM. For the POD--DEIM ROM 
such bound is given in the original paper~\cite{chaturantabut2010nonlinear}. In our notation the result reads: 
For some $\blf\in\mathbb{R}^M$ let $\widehat\blf= \rY_{\rm pod} (\rP \rY_{\rm pod})^{-1}\rP \blf$, then it holds
$
\|\blf-\widehat\blf\|\le \|(\rP \rY_{\rm pod})^{-1}\|\|( \rI- \rY_{\rm pod}( \rY_{\rm pod})^T)\blf\|.
$
A bound on $C_\ast(n)=\|(\rP \rY_{\rm pod})^{-1}\|$ was also derived in~\cite{chaturantabut2010nonlinear}.
Applying this result for \emph{in-sample} $\balpha \in \widehat{\cA}$, one finds the interpolation estimate
\begin{equation}
	\label{eqn:est3}
\max_{k=1,\dots,N} 
\left\|\blf^k- \widehat\blf^k\right\|_{\ell^2}\le C_\ast(n)\sigma_n(\Psi_{\rm pod}), 
\quad \mbox{with~} \widehat\blf^k= \rY_{\rm pod} (\rP \rY_{\rm pod})^{-1}\rP \blf^k.
\end{equation}
In the context of ROM for parametric systems, the estimate in \eqref{eqn:est3} shows the following limitations: 
(i) It is not clear how it can be extended for an out-of-sample $\balpha$, and 
(ii) For the case of higher variability w.r.t. parameters, the singular values of  $\Psi_{\rm pod}$ 
may decrease relatively slowly (see Figure~\ref{fig1}) requiring higher dimensions of reduced $\blf$-spaces.  
The interpolation bound for the two-stage TROM--DEIM given below addresses both of the above issues.

\begin{theorem}
	\label{Th1}
	For any given $\balpha \in \cA$,  let $ \rY_{\balpha}= \rY   \rY_n$,  $\rP_{\balpha}= \rP_n\rP$ for   the  local pointwise interpolation  and  $\rP_{\balpha}=\rP$ for the least-square local interpolation. Also let
	$\widehat\blf^k= \rY_{\balpha}(\rP_{\balpha} \rY_{\balpha})^{\dagger}\rP_{\balpha}\blf^k$, and
	$\sigma_i$ are singular values of $\widetilde\Psi(\balpha)$. It holds
	\begin{equation}
		\label{eqn:est2}
		\frac1{NM}\mbox{\small $\displaystyle \sum_{k=1}^{N}$} 
		\left\|\blf^k- \widehat\blf^k\right\|_{\ell^2}^2 
		\le \frac{C_\ast}{NM} \left( C_1 \eps^2\left\| \bPsi \right\|_F + 
		\mbox{\small $\displaystyle\sum_{i = n+1}^{N}$} \tsigma_i^2 \right) + C_2 \delta^{2p} ,
	\end{equation}
	with $C_\ast =\|(\rP_{\balpha} \rY_{\balpha})^{-1}\|$  for   the  local pointwise interpolation  and $C_\ast =\|(\rP \rY)^{-1}\|$ for the least-square local interpolation.
	Constants $C_1$, $C_2$ are independent of $\balpha$, sampling grid, local TROM dimension $n$, tensor ranks and FOM dimensions.
\end{theorem}
\begin{proof} 
Note that for the local DEIM the matrix $\rP_{\balpha} \rY_{\balpha} = \rP_n \rP \rY \rY_n$ is 
invertible and hence its inverse coincides with the pseudo-inverse, so we use the notation
$(\rP_{\balpha} \rY_{\balpha})^{\dagger}$ throughout the proof to cover both local hyper-reduction
variants. Clearly, for both variants $(\rP_{\balpha} \rY_{\balpha})^{\dagger}$ is the \emph{left} 
inverse of $\rP_{\balpha} \rY_{\balpha}$ and so we have the identity 
$(\rI- \rY_{\balpha}(\rP_{\balpha} \rY_{\balpha})^{\dagger}\rP_{\balpha}) \rY_{\balpha}=0$. 
We employ it to estimate
\begin{equation}\label{aux671}
\begin{split}
\|\blf^k- \widehat\blf^k\|_{\ell^2} & 
=\|(\rI- \rY_{\balpha}(\rP_{\balpha} \rY_{\balpha})^{\dagger}\rP_{\balpha})\blf^k\|_{\ell^2}
=\|(\rI- \rY_{\balpha}(\rP_{\balpha} \rY_{\balpha})^{\dagger}\rP_{\balpha})
(\rI- \rY_{\balpha} \rY_{\balpha}^T)\blf^k\|_{\ell^2}\\
&\le \|\rI- \rY_{\balpha}(\rP_{\balpha} \rY_{\balpha})^{\dagger}\rP_{\balpha}\|
\|(\rI- \rY_{\balpha} \rY_{\balpha}^T)\blf^k\|_{\ell^2}.
\end{split}
\end{equation}
For the local DEIM, the first factor on the right-hand side of \eqref{aux671} 
can be estimated by the same argument as in ~\cite{chaturantabut2010nonlinear}:
\begin{equation}\label{aux679}
\|\rI- \rY_{\balpha}(\rP_{\balpha} \rY_{\balpha})^{\dagger}\rP_{\balpha}\|
=\| \rY_{\balpha}(\rP_{\balpha} \rY_{\balpha})^{\dagger}\rP_{\balpha}\|\le 
\| \rY_{\balpha}\|\|(\rP_{\balpha} \rY_{\balpha})^{\dagger}\|\|\rP_{\balpha}\|= \|(\rP_{\balpha} \rY_{\balpha})^{\dagger}\|,
\end{equation}
where the first identity holds since 
$ \rY_{\balpha}(\rP_{\balpha} \rY_{\balpha})^{\dagger}\rP_{\balpha}$ is a projector. 

For the local LS fitting 
the matrix $ \rY_{\balpha}(\rP_{\balpha} \rY_{\balpha})^{\dagger}\rP_{\balpha}$ is still a projector.
Furthermore, for any $\bx\in\mathbb{R}^{\widetilde N}$, where $\widetilde N$ is the dimension of the 
universal reduced space $\widetilde Y$, and $\by=(\rP_{\balpha} \rY_{\balpha})^{\dagger}\bx$, 
it holds $\rP_{\balpha} \rY_{\balpha}\by=\rP'\bx$, where $\rP'$ is the orthogonal projector on the 
range of $\rP_{\balpha} \rY_{\balpha}$. 

Let $\widehat\bu=  \rY_n\by$, then we have 
$\rP \rY\widehat\bu = \rP_{\balpha} \rY\widehat\bu=\rP_{\balpha} \rY_{\balpha}\by=\rP'\bx$ implying 
\[
\|  \rY_n(\rP_{\balpha} \rY_{\balpha})^{\dagger}\bx\|=\|\widehat\bu\|=\|(\rP \rY)^{\dagger}\rP'\bx\|\le
\|(\rP \rY)^{\dagger}\|\|\rP'\bx\|\le\|(\rP \rY)^{\dagger}\|\|\bx\|.
\]
This yields $\|  \rY_n(\rP_{\balpha} \rY_{\balpha})^{\dagger}\|\le \|(\rP \rY)^{\dagger}\|$ 
and so for the local LS fitting we have 
\begin{equation}\label{aux679b}
	\begin{split}
		\|\rI- \rY_{\balpha}(\rP_{\balpha} \rY_{\balpha})^{\dagger}\rP_{\balpha}\|&=\| \rY_{\balpha}(\rP_{\balpha} \rY_{\balpha})^{\dagger}\rP_{\balpha}\|\le 
		\| \rY\|\|  \rY_n(\rP_{\balpha} \rY_{\balpha})^{\dagger}\|\|\rP_{\balpha}\|\\
		&= \|  \rY_n(\rP_{\balpha} \rY_{\balpha})^{\dagger}\|\le \|(\rP \rY)^{\dagger}\|.
	\end{split}
\end{equation}

Consider the SVD of $\widetilde{\Psi} (\balpha) \in \R^{M \times N}$ from \eqref{eqn:extractbt} given by
	\begin{equation}
		\widetilde{\Psi}(\balpha) = \widetilde{ \rY} \widetilde{\Sigma} \widetilde{\rZ}^T,
		~~\text{with}~~ \widetilde{\Sigma} = \text{diag} (\tsigma_1,\dots,\tsigma_N).
	\end{equation}
Then from the definition of the basis for the local reduced $\blf$-space it follows that 
$ \rY_{\balpha}= \rY  \rY_n = [ \by_1, \ldots,  \by_n] \in \R^{M \times n}$ 
are the first $n$ columns of $\widetilde{\rY}$. 
Let $\rF(\balpha)=[\blf^1,\dots,\blf^N]\in\mathbb{R}^{M\times N}$, then
	\begin{align}
		 \mbox{\small $\displaystyle\sum_{k=1}^{N} $}
		\left\|\blf^k- \widehat\blf^k\right\|_{\ell^2}^2 &=
		\left\| (\rI  -  \rY_{\balpha}  \rY_{\balpha}^T) \rF(\balpha) \right\|^2_F 
		\le  \left( \left\| (\rI -  \rY_{\balpha}  \rY_{\balpha}^T) (\rF(\balpha) - \widetilde{\Psi}(\balpha)) \right\|_F
		+ \left\| (\rI -  \rY_{\balpha}  \rY_{\balpha}^T) \widetilde{\Psi}(\balpha) \right\|_F \right)^2 \nonumber \\
		& \le \left( \left\| \rF(\balpha) - \widetilde{\Psi}(\balpha) \right \|_F 
		+ \left\| (\rI -  \rY_{\balpha}  \rY_{\balpha}^T) \widetilde{\Psi}(\balpha) \right\|_F \right)^2, 
		\label{eqn:aux413}
	\end{align}
where we used triangle inequality and $\| \rI -  \rY_{\balpha}  \rY_{\balpha}^T\| \le 1$ 
for the spectral norm of the projector.
For the last term in \eqref{eqn:aux413}, we observe
\begin{equation}\label{aux701}
	\left\| (\rI -  \rY_{\balpha}  \rY_{\balpha}^T) \widetilde{\Psi}(\balpha) \right\|_F =
	\left\| \widetilde{ \rY} \; \text{diag}(0,\dots,0,\tsigma_{n+1},\dots,\tsigma_N) \; \widetilde{\rZ}^T \right\|_F =
	\left( \mbox{\small $\displaystyle\sum_{j = n+1}^{N}$} \tsigma_j^2 \right)^{\frac12}.
	\end{equation}
	To handle the first term of \eqref{eqn:aux413}, consider the interpolation of the full (non-compressed) $\blf$-snapshot tensor
		$\Psi_I (\balpha) = \bPsi \times_2 \be^1 (\balpha) \times_3 \be^2(\balpha) \dots \times_{D+1} \be^D(\balpha)$,
	and proceed using the triangle inequality
	\begin{equation}
		\label{eqn:aux438}
		\big\| \rF(\balpha) - \widetilde{\Psi}(\balpha) \big\|_F \le 
		\big\| \rF (\balpha) - \Psi_I (\balpha) \big\|_F + 
		\big\| \Psi_I(\balpha) - \widetilde{\Psi} (\balpha) \big\|_F.
	\end{equation}
The polynomial interpolation procedure is stable in the sense that
\begin{equation}
	\label{eqn:int_stab}
	\left(\mbox{\small $\displaystyle\sum_{j=1}^{K_i}$} \left| e^i_j(a \be_i) \right|^2 \right)^{\frac12} \le C_e,
\end{equation}
where $\be_i \in \R^{K_i}$ is the $i^{\mbox{\scriptsize th}}$ column of the $K_i \times K_i$ identity matrix, 
and with some $C_e$ independent of $a \in [\alpha_i^{\min}, \alpha_i^{\max}]$ 
and $i = 1, \ldots, D$ (e.g., $C_e=1$ for linear interpolation).
We use this and \eqref{eqn:TensApprox} to bound the second term
	of \eqref{eqn:aux438}. Specifically,
	\begin{equation}
		\label{eqn:aux442}
		\begin{split}
			\left\| \Psi_I (\balpha) - \widetilde{\Psi}(\balpha) \right\|_F & =
			\left\| (\bPsi - \widetilde{\bPsi}) 
			\times_2 \be^1(\balpha) \times_3 \be^2(\balpha) \dots \times_{D+1} \be^D (\balpha) \right\|_F \\
			& \le \left\| \bPsi - \widetilde{\bPsi} \right\|_F 
			\| \be^1 (\balpha) \|_{\ell^2} \| \be^2(\balpha) \|_{\ell^2} \dots \| \be^D (\balpha) \|_{\ell^2} \\
			& \le (C_e)^D \left\| \bPsi - \widetilde{\bPsi} \right\|_F 
			\le (C_e)^D \eps\left\| \bPsi \right\|_F.
		\end{split}
	\end{equation}
	It remains to handle the first term in \eqref{eqn:aux438}. 	The choice of interpolation procedure in \eqref{eqn:lagrange} implies that  for any sufficiently smooth $f:  [\alpha_i^{\min}, \alpha_i^{\max}] \to \R$ it holds
	\begin{equation}
		\label{eqn:int_aprox}
		\sup_{a \in [\alpha_i^{\min}, \alpha_i^{\max}]} 
		\Big| f(a) - \mbox{\small $\displaystyle\sum_{j=1}^{K_i}$} e^i_j \left( a \be_i \right) f(\halpha_i^j) \Big|\le
		C_a \| f^{(p)} \|_{C( [\alpha_i^{\min}, \alpha_i^{\max}])} \delta_i^p, 
	\end{equation}
	for $i = 1, \dots, D$, where the constant $C_a$ does not depend on $f$. 
	
	Using the shortcut notation $\blf^k(\alpha_1,\dots, \alpha_D)=\blf(t_k,\bu(\alpha_1^, \dots, \alpha_D), \alpha_1, \dots, \alpha_D)$
and	interpolation property \eqref{eqn:int_aprox}, we compute
	\begin{equation*}
		\begin{split}
			\left( \bPsi \times_2 \be^1(\balpha) \right)_{:,i_2,\dots,i_D,k} & 
			= \mbox{\small $\displaystyle\sum_{j=1}^{K_1}$} e^1_j(\balpha) 
			\blf^k(\halpha_1^{j}, \halpha_2^{i_2}, \dots, \halpha_D^{i_D}) \\
			& = \blf^k(\alpha_1, \halpha_2^{i_2} \dots, \halpha_D^{i_D}) + \Delta^1_{:,i_2,\dots,i_D,k},
		\end{split}
	\end{equation*}
where $\bhalpha \in \widehat{\cA}$. The $\Delta$-term obeys a component-wise bound
	\begin{equation*}
		|\Delta^1_{:,i_2,\dots,i_D,k}| \le C_a \sup_{a \in [\alpha_1^{\min}, \alpha_1^{\max}]}
	 \left| \frac{\partial^p \blf^k}{\partial \alpha^p_1}(a, \halpha_2^{i_2}, \dots, \halpha_D^{i_D}) \right|\delta_1^p
	\end{equation*}
	where the absolute value of vectors is understood entry-wise.
	Analogously, we compute
	\begin{equation}
	\begin{split}
		\big(\bPsi \times_2 \be^1(\balpha) & \times_3 \be^2(\balpha) \big)_{:,i_3,\dots,i_D,k} = 
		\left( (\bPsi \times_2 \be^1(\balpha)) \times_2 \be^2(\balpha) \right)_{:,i_3,\dots,i_D,k}\\
		& = \mbox{\small $\displaystyle\sum_{j=1}^{K_2}$} e^2_j(\balpha) 
		\left( \blf^k(\alpha_1, \halpha_2^{j}, \halpha_3^{i_3}, \dots, \halpha_D^{i_D}) 
		+ \Delta^1_{:, j, i_3,\dots,i_D,k} \right) \\
		& = \blf^k(\alpha_1, \alpha_2, \halpha_3^{i_3}, \dots, \halpha_D^{i_D}) 
		+ \Delta^2_{:,i_3,\dots,i_D,k} + \mbox{\small $\displaystyle\sum_{j=1}^{K_2}e^2_j(\balpha)$} \Delta^1_{:,j,i_3,\dots,i_D,k},
	\end{split}
	\end{equation}
	with a component-wise bound for the remainder
	\begin{equation*}
		\begin{split}
			\Big| \Delta^2_{:, i_3, \dots, i_D, k} + 
			\mbox{\small $\displaystyle\sum_{j=1}^{K_2}$} e^2_j(\balpha)  \Delta^1_{:, j, i_3, \dots, i_D, k} \Big| 
			\le C_a & \sup_{a \in [\alpha_2^{\min}, \alpha_2^{\max}]} 
			\left| \frac{\partial^p \blf^k}{\partial \alpha^p_2}
			(\alpha_1, a, \halpha_3^{i_3}, \dots, \halpha_D^{i_D}) \right| \delta_2^p \\
			+ \; C_e \; C_a & \sup_{a \in [\alpha_1^{\min}, \alpha_1^{\max}]}
			\left| \frac{\partial^p \blf^k}{\partial \alpha^p_1}
			(a, \halpha_2^{i_2}, \dots, \halpha_D^{i_D}) \right| \delta_1^p.
		\end{split}
	\end{equation*}
	Applying the same argument repeatedly, we obtain
	\begin{equation}\label{aux778}
		\begin{split}
			\left( \Psi_I (\balpha) \right)_{:, k} & = 
			\left( \bPsi \times_2 \be^1(\balpha) \times_3 \be^2(\balpha) \dots \times_{D+1} \be^D(\balpha) \right)_{:, k} \\
			& = \blf^k (\alpha_1, \alpha_2, \dots, \alpha_D) 
			+ \Delta_{:, k} = \left( \rF(\balpha) \right)_{:, k} + \Delta_{:,k},
		\end{split}
	\end{equation}
	with a component-wise bound for the remainder
	\begin{align*}
		\left| \Delta_{:,k} \right| \le C_a \Big( & \sup_{a \in [\alpha_D^{\min}, \alpha_D^{\max}]}
		\left|\frac{\partial^p \blf^k}{\partial \alpha^p_D}
		(\alpha_1, \dots, \alpha_{D-1}, a) \right| \delta_D^p + \dots \\
		+\, (C_e)^{D-2} & \sup_{a \in [\alpha_2^{\min}, \alpha_2^{\max}]}
		\left| \frac{\partial^p \blf^k}{\partial \alpha^p_2}
		( \alpha_1, a, \halpha_3^{i_3}, \dots, \halpha_D^{i_D}) \right| \delta_2^p \\
		+\, (C_e)^{(D-1)} & \sup_{a \in [\alpha_1^{\min}, \alpha_1^{\max}]}
		\left| \frac{\partial^p \blf^k}{\partial \alpha^p_1}
		(a, \halpha_2^{i_2}, \dots, \halpha_D^{i_D}) \right| \delta_1^p \Big)\le C_a \max \left\{ (C_e)^{(D-1)}, 1 \right\}C_{\blf}\delta^p.
	\end{align*}
Using the definition of the Frobenius norm and \eqref{aux778},  we arrive at
	\begin{equation}
		\label{eqn:aux450}
		\left\| \rF(\balpha) - \Psi_I (\balpha) \right\|_F \le 
		\sqrt{N M} \; C_a \max \left\{ (C_e)^{(D-1)}, 1 \right\}C_{\blf}\delta^p,
	\end{equation}
with constants $C_a$, $C_e$, $C_{\blf}$ from \eqref{eqn:int_aprox}, \eqref{eqn:int_stab}, 
and \eqref{eqn:setF}, respectively. 
Combining \eqref{aux671}--\eqref{aux679b}, \eqref{eqn:aux413}--\eqref{eqn:aux438}, \eqref{eqn:aux442} 
and \eqref{eqn:aux450} proves the final result.
\end{proof}

It is interesting to note that for the local LS fitting variant of the two-stage hyper-reduction the 
constants in the interpolation Theorem~\ref{Th1} are all independent of $\balpha$. 
Moreover, the constant $C_\ast$ is essentially the same as in the original estimate~\eqref{eqn:est3}.

\section{Numerical experiments}
\label{s:num}

We perform several numerical experiments to assess the performance of the TROM
and compare it to the POD--DEIM ROM. Our testing  is performed for dynamical systems 
arising from discretizations of the parameter-dependent Burgers and Allen-Cahn equations. 
All results presented below are for the local least squares second stage of the hyper-reduction. 
The local DEIM as the second stage were found to yield very close results. 

\subsection{Parameterized 1D Burgers equation}
\label{sec:burgers}

As a first example consider the 1D Burgers equation:
Find $u(t,x)$, solving
\begin{equation}
u_t = \alpha_1 u_{xx}- uu_x, \quad \text{for}~x\in(0,1),~t\in(0,T), \quad u(t,0)=u(t,1)=0,
\label{eqn:Burgers}
\end{equation}
where $\alpha_1 > 0$ is the viscosity parameter. The initial condition is in parametric form
\begin{equation}
u_{\balpha}(0, x) = u_0(x, \alpha_2)=\left\{
\begin{array}{rl}
	1,&\quad x\in(0,\alpha_2)\\
	0,&\quad x\in[\alpha_2,1)
\end{array}
\right.,\quad \text{for}~\alpha_2 \in (0,1).
\label{eqn:BIC}
\end{equation}
 Thus, we consider a two-parameter system,  with the parameter 
domain $\cA = [0.01, 0.5] \times [0.2,0.8]$.

We discretize \eqref{eqn:Burgers}--\eqref{eqn:BIC} in space using a first order upwind 
finite difference (FD) scheme on a uniform grid with mesh size $h=1/M$ to obtain a dynamical system 
of the form \eqref{eqn:GenericPDE}, where $\rA_{\balpha} \in \mathbb{R}^{M \times M}$ 
depends on $\alpha_1$ only and the non-linear term $\blf_{\balpha}(t, \bu) = \blf(\bu)$ 
is the discretization of $-uu_x$ that does not  contain any dependence on $\balpha$.

To generate FOM snapshots we integrate in \eqref{eqn:GenericPDE} time using a semi-implicit 
BDF2 method with equidistant time stepping: for each $\widehat{\balpha}_k \in \widehat{\cA}$,
given $\bphi_j(\widehat{\balpha}_k)$ and $\bphi_{j-1}(\widehat{\balpha}_k)$ in $\mathbb{R}^M$, 
find $\bphi_{j+1}(\widehat{\balpha}_k)$ 
satisfying
\begin{equation}
\frac{3 \bphi_{j+1}(\widehat{\balpha}_k) - 4 \bphi_j(\widehat{\balpha}_k ) + 
\bphi_{j-1}(\widehat{\balpha}_k)}{2 \Delta t} =
\rA_{\widehat{\balpha}_k} \bphi_j(\widehat{\balpha}_k) 
- \left( 2 \bphi_j(\widehat{\balpha}_k) - \bphi_{j-1}(\widehat{\balpha}_k) \right) \odot
\rG_x \bphi_{j+1}(\widehat{\balpha}_k),
\label{eqn:BDF2}
\end{equation}
with first step performed via BDF1.  
Here $j = 1,2,\ldots,N$, $\Delta t= T/N$, $T=1$, and $k=1,\ldots,K$. 
We denote by $\rG_{x} \in \mathbb{R}^{M \times M}$ the matrix discretization of the first 
derivative, with boundary conditions \eqref{eqn:Burgers}. 
The operation $\odot$ denotes entrywise product of vectors in $\mathbb{R}^M$.

Once the time-stepping \eqref{eqn:BDF2} is computed, $\blf$-snapshots are given by 
\begin{equation}
\bpsi_j(\widehat{\balpha}_k) = - \left( 2 \bphi_j(\widehat{\balpha}_k) - \bphi_{j-1}(\widehat{\balpha}_k) \right) 
\odot \rG_x \bphi_{j+1}(\widehat{\balpha}_k), \quad 
j = 1,2,\ldots,N, \quad k=1,\ldots,K.
\end{equation}

\begin{table}[h]
\caption{ 
	Compression ranks $[\widetilde R_1^\Theta, \widetilde R_2^\Theta, \widetilde R_3^\Theta]$, 
	$\Theta \in \{ \Phi, \Psi\}$, for TT-TROM versus compression accuracy $\eps$ and the corresponding 
	compression factors (top). Compression accuracy for CP-TROM versus canonical rank (CP-rank) and the 
	corresponding compression factors (bottom).}
\label{tab1}
\small
\begin{center}
\begin{tabular}{|l|c|c|c|c|c|}
\hline
\multicolumn{1}{|c|}{${\eps}$} & 0.1 & 0.03 & {0.01} & 0.003 & $10^{-3}$ \\
\hline		
$\widetilde{\bPhi}$ TT-ranks        & $[\textbf{7},9,\textbf{4}]$  & $[\textbf{12},20,\textbf{7}]$ & 
$[\textbf{20},42,\textbf{10}]$  & $[\textbf{33},78,\textbf{14}]$  & $[\textbf{44},121,\textbf{17}]$  \\
$\widetilde{\bPsi}$ TT-ranks       & $[\textbf{55},55,\textbf{10}]$ & $[\textbf{93},101,\textbf{16}]$ & 
$[\textbf{139},151,\textbf{20]}$ & $[\textbf{174},210,\textbf{25}]$ & 
{\footnotesize$[\textbf{200},265,\textbf{30}]$} \\
$\mbox{CF}(\bPhi)$     &  18824 & 4894 & 1518 & 536 & 270 \\
$\mbox{CF}(\bPsi)$       &  620 & 202  & 95 & 54 &37 \\
\hline \hline	
\multicolumn{1}{|c|}{CP-rank}     & 50   &100 & 200 &300 & 500  \\  \hline	
$\|\widetilde{\bPhi}-\bPhi\|/\|\bPhi\|$ & $3.22 \cdot 10^{-2}$ & $2.13 \cdot 10^{-2}$ & $1.26 \cdot 10^{-2}$ &
$9.13 \cdot 10^{-3}$ & $5.96 \cdot 10^{-3}$   \\
$\|\widetilde{\bPsi}-\bPsi\|/\|\bPsi\|$ & $1.63 \cdot 10^{-1}$ & $9.72 \cdot 10^{-2}$ & $5.74 \cdot 10^{-2}$ &
$4.12 \cdot 10^{-2}$ & $2.84 \cdot 10^{-2}$ \\
$\mbox{CF}(\bPhi, \bPsi)$ & 8274  & 2749  &  822  & 391 & 149 \\
\hline				
\end{tabular}
\end{center}
\end{table}

\subsubsection{Tensor compression accuracy and TROM solution quality}
\label{sec:tcompacc}

In the first series of experiments we study the performance of TROMs depending on the LRTD accuracy $\eps$. 
We set $N=200$, $M=400$ and use a rather 
fine grid of $32$ uniformly distributed $\widehat{\alpha}_2^j$ values in $[0.2,0.8]$, and $16$ values of 
$\widehat{\alpha}_1^j$ log-uniformly distributed in $[0.01,0.5]$ to define the sampling set $\widehat{\cA}$.

For TT-TROM we start with a low accuracy of ${\eps}=0.1$ (we always set the same target accuracy for 
$\widetilde{\bPhi}$ and $\widetilde{\bPsi}$) and gradually improve it letting  
${\eps}=\{0.1,0.03,0.01, 0.003, 10^{-3}\}$. We report in Table~\ref{tab1} the resulting TT compression 
ranks in the format $[\widetilde R_1^\Theta, \widetilde R_2^\Theta, \widetilde R_3^\Theta]$ for both 
$\Theta \in \{ \Phi, \Psi\}$, and the compression factor 
\begin{equation}
\mbox{CF}(\bTheta) = \#(\bTheta)/\#\mbox{online}(\widetilde{\bTheta}),
\end{equation}
for $\bTheta \in \{ \bPhi, \bPsi \}$, where $\#(\bTheta)$ is the number of entries of the full snapshot
tensor $\bTheta$ and $\#\mbox{online}(\widetilde{\bTheta})$ is the amount of data passed in
Algorithm~\ref{alg:TROM-DEIM} from the offline stage to the online stage. For CP-TROM we used
the combined compression factor 
\begin{equation}
\mbox{CF}(\bPhi, \bPsi) = 
\frac{\#(\bPhi) + \#(\bPsi)}{\#\mbox{online}(\widetilde{\bPhi})  + \#\mbox{online}(\widetilde{\bPsi})}.
\end{equation}
The first and last TT ranks emphasized in bold are responsible for the universal reduced space dimension 
and the maximum possible local reduced space dimension, respectively. For the same $\eps$, the first and last 
Tucker ranks (not shown) of HOSVD-TROM were found to be the same as the first and last TT ranks, but the 
HOSVD-TROM compression factor (not shown) was slightly worse than those of TT-TROM.  
The  higher TT ranks for $\tbPsi$ than for $\tbPhi$ indicate a higher variability in $\blf$-snapshots compared 
to that of $\bu$-snapshots. We shall see that this high variability  results into a 
dramatic accuracy gain of TROMs compared to conventional POD--DEIM ROM.  

\begin{figure}[h]
	\begin{center}\footnotesize
	\begin{tabular}{cc}
	TT-TROM & CP-TROM \\
		\includegraphics[width=0.425\textwidth]{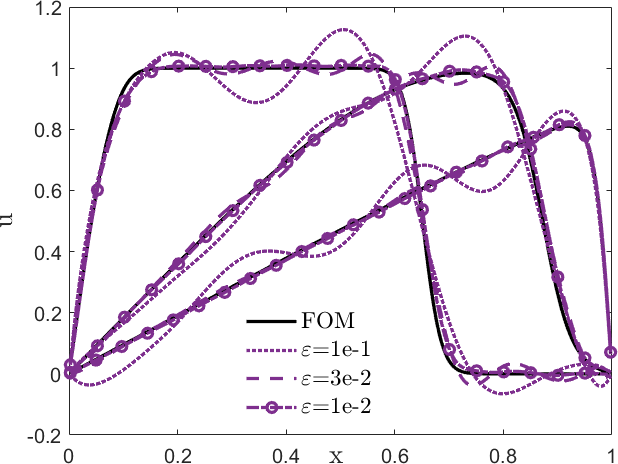} &
		\includegraphics[width=0.425\textwidth]{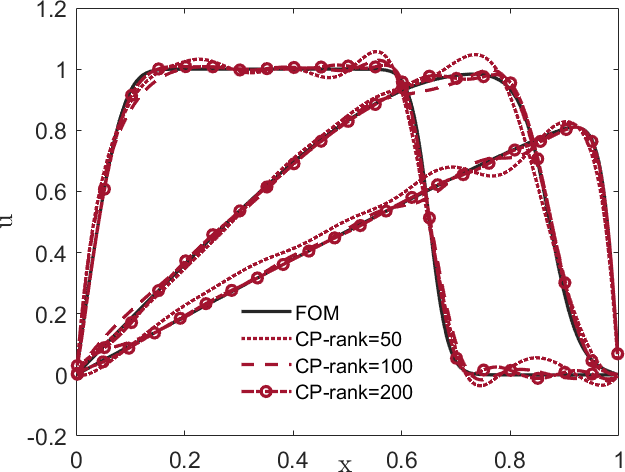} 
		\end{tabular}
	\end{center}
	\caption{FOM and TROM solutions of the discretized Burgers equation at times $t={0.05,0.5,1}$ 
	for out-of-sample parameter values $\alpha_1=0.013$ and $\alpha_2=0.633$.}
	\label{fig:Eps}
\end{figure}

To assess the effect of tensor compression accuracy on the TROM solutions, we selected an out-of-sample parameter vector $\balpha=(0.013, 0.633)^T \notin \widehat{\cA}$ and computed the corresponding TT-TROM solution at $t={0.05, 0.5, 1}$. We set $n_{\Theta} = \widetilde{R}^{\Phi}_{D+1}$ for $\Theta \in { \Phi, \Psi }$, which corresponds to using the local reduced spaces of the largest possible dimension; see the bound in \eqref{eqn:nphipsi}. The resulting solution is displayed in Figure~\ref{fig:Eps} (left), where it is compared to the FOM solution. The HOSVD-TROM solutions were virtually identical to those provided by TT-TROM, so we do not display them separately.
We observe that TT-TROM solutions for $\epsilon=0.1$ exhibit significant inaccuracies, but the results improve rapidly for smaller values of $\epsilon$. For $\epsilon=10^{-2}$, the TT-TROM solutions closely align with the FOM solution. We do not present the TT-TROM solutions for $\epsilon={0.003, 10^{-3}}$ since they are visually indistinguishable from the FOM solution.

In Figure~\ref{fig:Eps} (right), we present the results of a similar study for CP-TROM, where the CP-rank is predetermined instead of targeting a specific accuracy $\epsilon$. The experiment covers CP-ranks of ${50, 100, 200, 300, 500}$ (further increasing the CP-rank leads to a very slow convergence of the ALS method~\cite{ReviewTensor} for computing CP LRTD). The corresponding compression accuracy of $\widetilde{\bPhi}$ and $\widetilde{\bPsi}$ is detailed in the lower part of Table~\ref{tab1}.

We note that the best compression accuracy achieved by CP-TROM for $\bPsi$ is only $2.82\cdot10^{-2}$ (for CP-rank=$500$). In Figure~\ref{fig:Eps} (right), CP-TROM solutions are displayed for CP-ranks of $50$ and $100$, with $n_\Phi = 5$ and $n_\Psi = 10$. For CP-rank=$200$, the dimensions of the local reduced spaces are chosen to be $n_\Phi = 7$ and $n_\Psi = 16$. This choice of $n_\Phi$ and $n_\Psi$ aligns with the TT-TROM local reduced space dimensions of comparable compression accuracy.

\begin{figure}[h]
	\begin{center}\footnotesize
	\begin{tabular}{cc}
	$n_\Phi = 5$, $n_\Psi = 10$ & $n_\Phi = 10$, $n_\Psi = 20$ \\
	\includegraphics[width=0.425\textwidth]{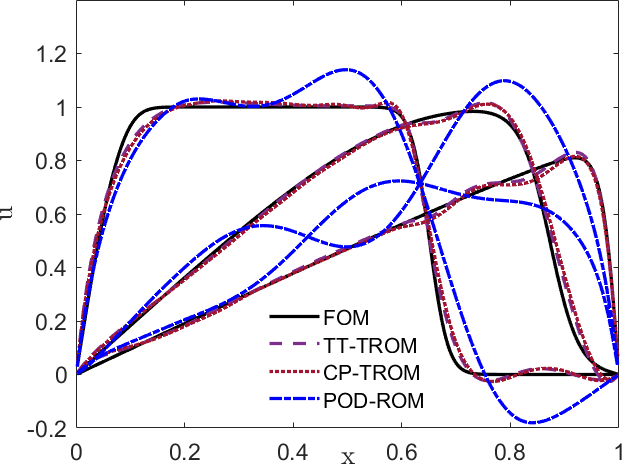} &
	\includegraphics[width=0.425\textwidth]{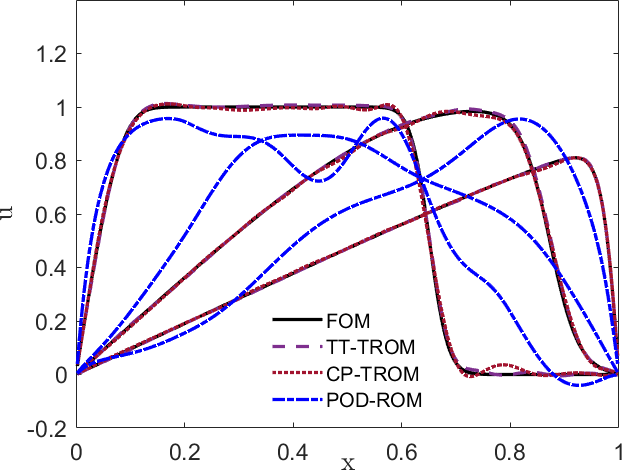} \\
	$n_\Phi = 15$, $n_\Psi = 30$ & POD--DEIM ROM for $\{ n_\Phi, n_\Psi \}$ \\
	\includegraphics[width=0.425\textwidth]{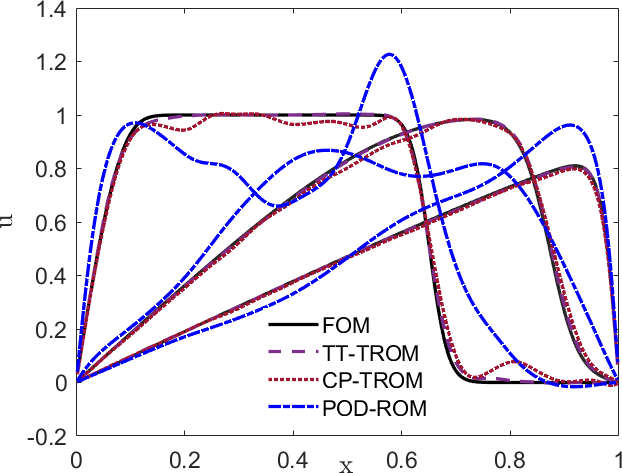} &
	\includegraphics[width=0.425\textwidth]{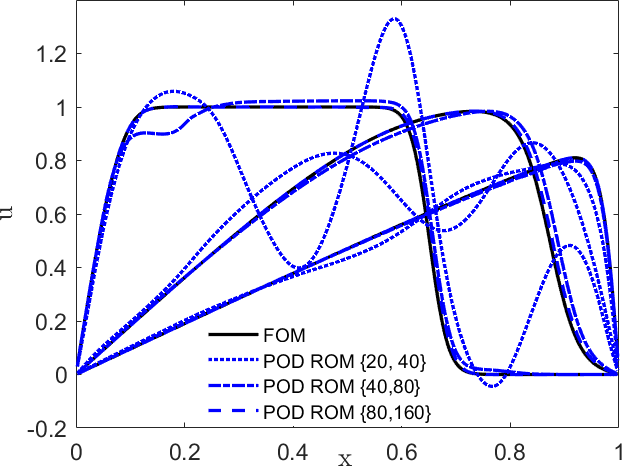} 		
	\end{tabular}
	\end{center}
	\caption{FOM and TROM solutions of the Burgers equation at times $t={0.05,0.5,1}$ 
	for out-of-sample parameter values $\alpha_1=0.013$ and $\alpha_2=0.633$, and increasing dimensions 
	of the local $\bu$- and $\blf$-spaces, $n_\Phi$ and $n_\Psi$, respectively.
	The right-bottom plot shows POD--DEIM ROM solutions for even higher 
	values of $n_\Phi$ and $n_\Psi$.}
	\label{fig:NM}
\end{figure}

\subsubsection{Out-of-sample TROM performance}

In this series of experiments we examine the performance of TROMs for out-of-sample parameter 
values depending on the dimensions of the local $\bu$- and $\blf$-spaces, denoted by $n_\Phi$ and $n_\Psi$,
respectively. We use the same sampling set $\widehat{\cA}$ as in Section~\ref{sec:tcompacc} and
the same out-of-sample parameter vextor $\balpha=(0.013, 0.633)^T \notin \widehat{\cA}$.
For both $\bPhi$ and $\bPsi$ tensors, we perform TT LRTD with $\eps=10^{-5}$ and CP LRTD with CP-rank=$200$ resulting in CP approximation accuracy of 
$\|\bPhi-\widetilde{\bPhi}_{\rm CP}\|_F = 1.3 \cdot 10^{-2}$ and 
$\|\bPsi-\widetilde{\bPsi}_{\rm CP}\|_F = 5.6 \cdot 10^{-2}$.

In Figure~\ref{fig:NM} we display both FOM and ROM solutions of the discretized Burgers equation at times 
$t={0.05,0.5,1}$, for increasing values of $n_\Phi$ and $n_\Psi$. We observe that already for
$n_\Phi = 5$ and $n_\Psi = 10$ both TT- and CP-TROM deliver reasonable approximations to the 
FOM solution. Increasing local reduced space dimensions to $n_\Phi = 10$ and $n_\Psi = 20$,
results in TROM solutions that almost match the FOM solutions. Remarkably, the further increase 
to $n_\Phi = 15$ and $n_\Psi = 30$ leads to CP-TROM solutions with some spurious oscillations, 
while TT-TROM solution provides a highly accurate approximation to the FOM solution. We attribute 
this degrade of CP-TROM to the failure of CP to approximate the snapshot tensors well for CP-rank=$200$ 
resulting in spurious higher order singular vectors in the bases $\rU_n$ and/or $\rY_n$. 
We display in all plots in Figure~\ref{fig:NM} solutions obtained with the conventional POD--DEIM ROM 
approach for the same dimensions $n_\Phi$ and $n_\Psi$. For all the dimensions considered 
POD--DEIM ROM fails to capture the behavior of the FOM solution. The right-bottom plot suggests 
that only for dimensions close to FOM size, it is possible for POD--DEIM ROM to ensure high quality
of the solution, which  defeats the purpose of model reduction.

\begin{figure}[h]
\begin{center}\footnotesize
	\begin{tabular}{cc}
	$E_\Phi$ & $E_\Psi$ \\
	\includegraphics[width=0.425\textwidth]{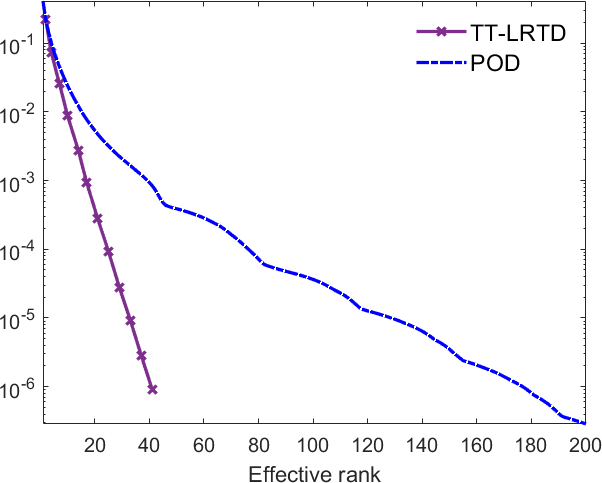} &
	\includegraphics[width=0.425\textwidth]{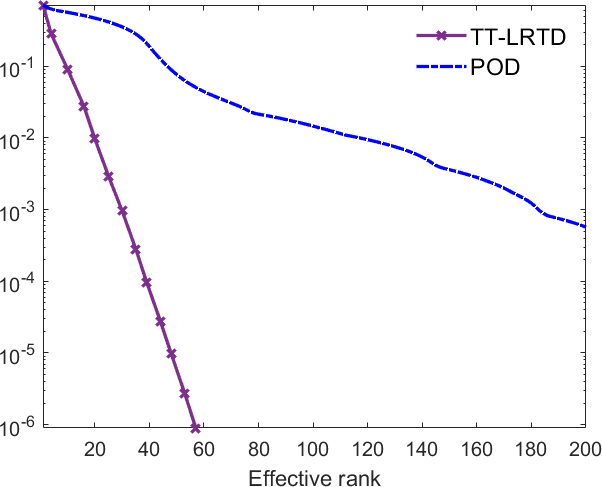} 		
	\end{tabular}
\end{center}
\caption{Relative errors of TT-LRTD snapshot tensor approximations \eqref{eqn:ELRTD} 
and POD \eqref{eqn:ESVD} versus their effective ranks.}
\label{fig:SVD}
\end{figure}

The dramatic gain in the ROM solution quality offered by the TROM compared to POD--DEIM ROM, 
as observed in Figure~\ref{fig:NM}, can be understood by exploring the relative approximation error of 
LRTD versus POD in terms of their \emph{effective rank}.
The effective ranks for TT- and HOSVD-LRTD are defined as $\widetilde R_{D+1}$ from \eqref{eqn:TTv}
and $\widetilde{N}$ from \eqref{eqn:TDv} respectively. For the truncated SVD of POD--DEIM ROM the 
effective rank is just the number of singular values/vectors kept in $\widetilde{\Theta}_{(1)}$, 
$\Theta \in \{ \Phi, \Psi \}$.

For LRTD employed in TROM we define the 
relative approximation errors for the $\bu$- and $\blf$- snapshot tensors as
\begin{equation}
E_\Phi = { \| \bPhi - \widetilde{\bPhi} \|_F }/{ \| \bPhi \|_F }, \quad
E_\Psi = { \| \bPsi - \widetilde{\bPsi} \|_F }/{ \| \bPsi \|_F }
\label{eqn:ELRTD}
\end{equation}
respectively. For the truncated SVD employed by POD--DEIM ROM the relative errors are
\begin{equation}
E_\Phi = { \| \Phi_{(1)} - \widetilde{\Phi}_{(1)} \|_F }/{ \| \Phi_{(1)} \|_F }, \quad
E_\Psi = /{ \| \Psi_{(1)} - \widetilde{\Psi}_{(1)} \|_F }/{ \| \Psi_{(1)} \|_F },
\label{eqn:ESVD}
\end{equation}
where $\Theta_{(1)}$ are the 1-mode unfolding matrices of the snapshot tensors $\bTheta$,
and $\widetilde{\Theta}_{(1)}$ are their truncated SVDs for both $\Theta \in \{ \Phi, \Psi \}$.
In practice, TT and HOSVD effective ranks 
are essentially the same, so we only report the errors \eqref{eqn:ELRTD} for TT-LRTD. 

Figure~\ref{fig:SVD} shows that for the same effective rank TT-LRTD provides a significantly
smaller relative error compared to the truncated SVD for both $\bu$- and $\blf$-snapshot
approximations. Moreover, the accuracy gain is especially pronounced for approximating
$\blf$-snapshots.

\begin{figure}[h]
	\begin{center}\footnotesize
	\begin{tabular}{cc}
	TT-TROM & CP-TROM \\
	\includegraphics[width=0.4\textwidth]{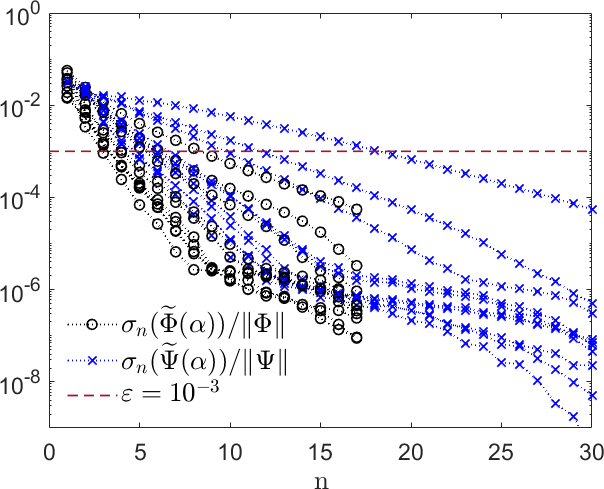} &
	\includegraphics[width=0.4\textwidth]{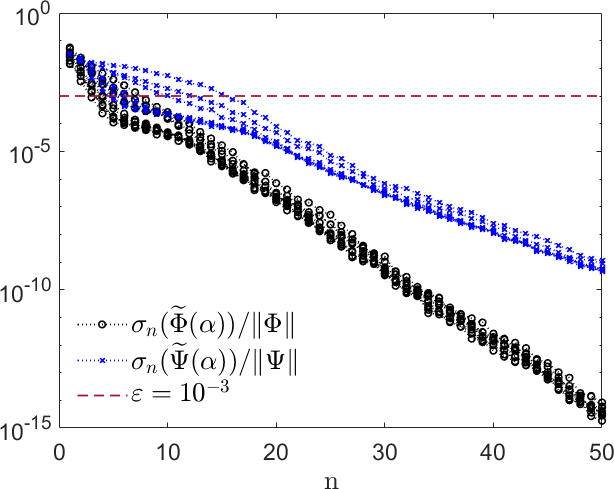} 		
	\end{tabular}
	\end{center}
	\caption{Scaled singular values of local spansot matrices $\widetilde{\Phi}(\balpha)$ and 
	$\widetilde{\Psi}(\balpha)$ defined in \eqref{eqn:extractbt} for $10$ random realizations of 
	out-of-sample parameters $\balpha \notin \widehat{\cA}$ for TT-TROM with 
	$\varepsilon=10^{-3}$ and CP-TROM with CP-rank=$200$, 
	showing $\sigma_n(\widetilde{\Theta}(\balpha))$, for $n=1,\ldots,50$, 
	$\widetilde \Theta \in \{ \widetilde \Phi, \widetilde \Psi \}$.}
	\label{fig:locSVD}
\end{figure}

Results in Figure~\ref{fig:NM} suggest that TT-TROM solutions are reasonably accurate even for the 
local dimensions lower than those given by the last TT rank: $n_{\Theta}<\widetilde R_{D+1}^{\Theta}$,  
$\Theta \in \{ \Phi, \Psi \}$. This behavior is explained by the left plot in Figure~\ref{fig:locSVD},
which shows the scaled singular values of the local snapshot matrices $\widetilde{\bPhi}(\balpha)$ 
and $\widetilde{\bPsi}(\balpha)$ for several random realizations of out-of-sample parameters 
$\balpha \notin \widehat{\cA}$. Note that $\sigma(\widetilde\bPhi(\balpha))$ and $\sigma(\widetilde\bPsi(\balpha))$ 
appear on the right-hand side of the TROM representation estimate \eqref{eqn:est1} 
and DEIM interpolation estimate~\eqref{eqn:est2}. While setting $n_{\Theta}= \widetilde R_{D+1}^{\Theta}$ 
makes the corresponding terms in \eqref{eqn:est1} and \eqref{eqn:est2} vanish, for smaller (but not too small) 
local reduced space dimensions they are dominated by approximation terms and so do not affect the representation 
power of the local reduced spaces.  

\begin{table}[h]
\caption{TROM errors in $H^1$-norm for refined parameter space grid. $\cA_{100}$ is a set of 100 pairs of 
randomly drawn out-of-sample parameters.}
\label{tab2}
\begin{center}\small
\begin{tabular}{|l|c|c|c|c|c|}
\hline
\multicolumn{1}{|c|}{$K_1\times K_2$} & $2 \times 4$ & $4 \times 8$ & $8 \times 16$ & 16x32 &32x64 \\
\hline		
$\frac{\sum\limits_{\boldsymbol{\alpha}\in\cA_{100}}\int\limits_{0.5}^1\int\limits_0^1
|(u_{\rm true}-u_{\rm pod})_x|^2dx\,dt}
{\sum\limits_{\boldsymbol{\alpha}\in\cA_{100}}\int\limits_{0.5}^1\int\limits_0^1
|(u_{\rm true})_x|^2dx\,dt}$ & 0.244 & 0.104 & 0.035 & 0.025 & 0.026 \\[1ex]
\hline
$\frac{\max\limits_{\boldsymbol{\alpha}\in\cA_{100}}\int\limits_{0.5}^1\int\limits_0^1
|(u_{\rm true}-u_{\rm pod})_x|^2dx\,dt}
{\max\limits_{\boldsymbol{\alpha}\in\cA_{100}}\int\limits_{0.5}^1\int\limits_0^1
|(u_{\rm true})_x|^2dx\,dt}$ &  0.442 & 0.204  & 0.060 & 0.030 & 0.033 \\
\hline 	
\end{tabular}
\end{center}
\end{table}

Finally we examine the dependence of TROM solution accuracy  on the sampling refinement for the parameter space.
Table~\ref{tab2} shows  errors in the $L^2(H^1)$ norm, which is a natural norm for parabolic problems.  
The error is  averaged  and maximal over 100 pairs of  out-of-sample parameters randomly drawn. 
We see that the error decays with refining the parameter grid  as predicted by the representation estimate 
with saturation achieved for $16 \times 32$. 

\subsection{Parameterized 2D Allen--Cahn equation}
\label{sec:AC}

The second example we consider is a 2D Allen--Cahn equation, a phase-field model of a phase separation process 
in $\Omega$ with a transition between order and disorder states. The model characterizes state of matter at 
$\bx\in\Omega$ by a smooth indicator function $u(t,\bx)$, solving 
\begin{equation}
u_t = \alpha_1^2 \Delta u - f(u), \quad \text{for}~\bx\in(0,1)^2,~t\in(0,T), 
\label{eqn:AC}
\end{equation}
with zero Neumann boundary conditions and the initial condition described below.
The nonlinear term is $f(u)=F'(u)$, where $F(u)=u^2(1-u)^2+\frac{\alpha_2}{10}(u^4-\frac12u)$ 
is a double-well Ginzburg–Landau potential to allow for phase separation. The parameter 
$\alpha_1 >0$ is the characteristic width of the transition region between the two phases, 
and $\alpha_2 \in [0,1]$ defines energy levels of pure phases with $\alpha_2 > 0$ corresponding 
to an asymmetric potential. 

To obtain a dynamical system of the form \eqref{eqn:GenericPDE}, we discretize \eqref{eqn:AC}
using a standard second-order finite difference scheme on a uniform grid with $M=150^2$ nodes and 
mesh size $h=1/\sqrt{M}$. This yields $\rA_{\balpha} \in \mathbb{R}^{M \times M}$ that depends on 
$\alpha_1$ only. The non-linear term $ \blf_{\balpha}(\bu) = F'(\bu)$
 depends on $\balpha$ and is being computed with the powers of $\bu \in \mathbb{R}^M$
taken entrywise.

To generate FOM snapshots for \eqref{eqn:AC} we apply a semi-implicit second-order time stepping 
scheme with uniform time step $\Delta t$: for each $\widehat{\balpha}_k \in \widehat{\cA}$,
given $\bphi_j(\widehat{\balpha}_k)$ and $\bphi_{j-1}(\widehat{\balpha}_k)$ in $\mathbb{R}^M$, 
find $\bphi_{j+1}(\widehat{\balpha}_k)$ satisfying
\begin{equation}
\begin{split}
\frac{3 \bphi_{j+1}(\widehat{\balpha}_k) - 4 \bphi_j(\widehat{\balpha}_k) 
	+ \bphi_{j-1}(\widehat{\balpha}_k) }{2 \Delta t}  + 
\beta_s \left( \bphi_{j+1}(\widehat{\balpha}_k) - 2 \bphi_j(\widehat{\balpha}_k) 
+ \bphi_{j-1}(\widehat{\balpha}_k) \right) = \\
\rA_{\widehat{\balpha}_k} \bphi_{j+1}(\widehat{\balpha}_k) 
- \blf_{\widehat{\balpha}} \left( 2 \bphi_j(\widehat{\balpha}_k) - \bphi_{j-1}(\widehat{\balpha}_k) \right), 
\end{split}
\label{eqn:ACtstep}
\end{equation}
for $j = 1,2,\ldots,N$, $\Delta t= T/N$, $T=20$, $N=200$ and $k=1,\ldots,K$. Here $\beta_s>0$ 
is the stabilization parameter~ \cite{shen2010numerical} to allow the explicit treatment 
of the non-linear term. 

\begin{figure}[h]
	\begin{center}\footnotesize
	\begin{tabular}{ccc}
	$p_{\rm b} = 0.50$ & $p_{\rm b} = 0.51$ & $p_{\rm b} = 0.52$ \\
		\includegraphics[width=0.31\textwidth]{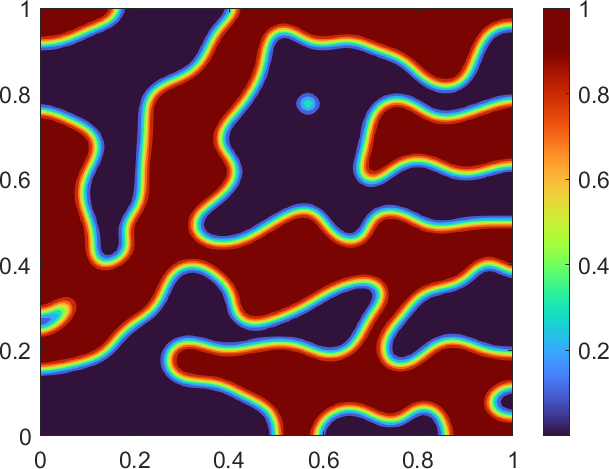} &	
		\includegraphics[width=0.31\textwidth]{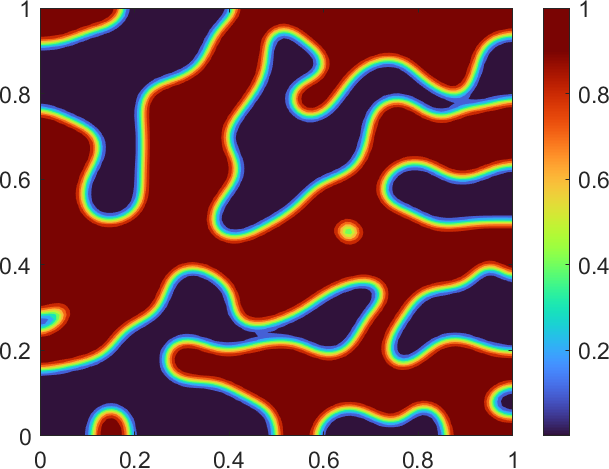} &
		\includegraphics[width=0.31\textwidth]{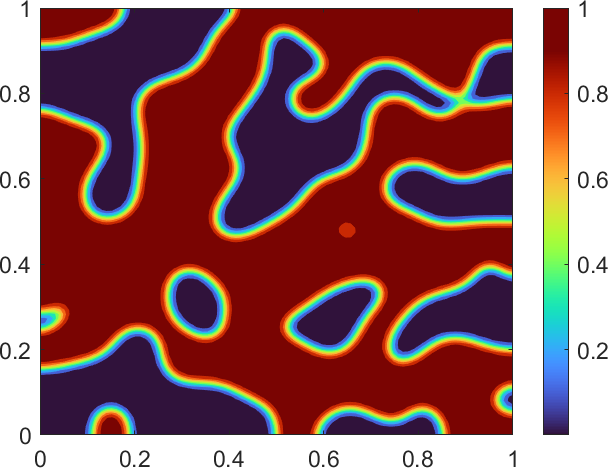} 		
	\end{tabular}
	\end{center}
	\caption{Initial conditions $\bphi_0(\widehat{\balpha})$ for three different values of 
	$p_{\rm b} \in \{0.50, 0.51, 0.52\}$.}
	\label{fig:in}
\end{figure}

The first step of \eqref{eqn:ACtstep} is performed using BDF1 for the initial
condition $\bphi_0(\widehat{\balpha}_k)$ with a varying ratio of the areas occupied by each 
state, as shown in Figure~\ref{fig:in}. The initial snapshots $\bphi_0$ are computed themselves
with FOM simulations of the discretized equation \eqref{eqn:AC} on $t \in (0,1)$ with 
$\alpha_1=0.01$ and $\alpha_2=0$ from the initial random Bernoulli distribution with probability of 
$u = 1$ equal to $p_{\rm b} \in \{0.50, 0.51, 0.52\}$, respectively, at each of $M$ nodes of the spatial 
discretization grid. 
Thus, we consider a three-parameter system, i.e., $D = 3$, 
with the parameter domain $\cA = [0.01, 0.025] \times [0,1] \times [0.5,0.52]$.

Once the $\bu$-snapshots are calculated via \eqref{eqn:ACtstep}, $\blf$-snapshots are simply 
\begin{equation}
\bpsi_j(\widehat{\balpha}_k) = - \blf_{\widehat{\balpha}} \left( 2 \bphi_j(\widehat{\balpha}_k) 
- \bphi_{j-1}(\widehat{\balpha}_k) \right), \quad j = 1,2,\ldots,N, \quad k=1,\ldots,K.
\end{equation}
We use a grid of $8$ log-uniformly distributed $\widehat{\alpha}_1^j$ values in $[0.01,0.025]$, 
three values of $\widehat{\alpha}_2^j \in\{0,0.15,0.3\} $ and three values of 
$\widehat{\alpha}_3^j \in\{0.50,0.51,0.52\}$ to define the sampling set $\widehat{\cA}$.

\begin{table}[h]
\label{tab4}
\caption{Compression ranks 
$[\widetilde R_1^\Theta, \widetilde R_2^\Theta, \widetilde R_3^\Theta, \widetilde R_4^\Theta ]$,
$\Theta \in \{ \Phi, \Psi\}$, for TT-TROM versus compression accuracy $\eps$ 
and the corresponding compression factors.}
	\vspace{-3ex}
	\footnotesize
	\begin{center}
	\begin{tabular}{|l|c|c|c|c|c|}
	\hline
	\multicolumn{1}{|c|}{${\eps}$} & 0.1 & 0.01 &$10^{-3}$& $10^{-4}$ & $10^{-5}$ \\
	\hline		
$\widetilde{\bPhi}$ TT-ranks & 
$[\textbf{13},6,3,\textbf{2}]$  & 
$[\textbf{79},35,13,\textbf{6}]$  & 
$[\textbf{221},90,35,\textbf{15}]$  & 
$[\textbf{442},155,63,\textbf{26}]$  & 
$[\textbf{741},229,94,\textbf{38}]$  \\
$\widetilde{\bPsi}$ TT-ranks & 
$[\textbf{120},46,16,\textbf{8}]$ & 
$[\textbf{340},129,47,\textbf{20}]$ & 
$[\textbf{646},216,87,\textbf{35}]$ & 
$[\textbf{1013},307,127,\textbf{50}]$ & 
$[\textbf{1405},399,165,\textbf{65}]$ \\
$\mbox{CF}(\bPhi)$ & $4.629\cdot10^5$ & $1.364\cdot10^5$ & 1902 & 555 & 226 \\
$\mbox{CF}(\bPsi)$ & 6921 & 870  & 274 & 123 &69 \\
	\hline				
	\end{tabular}
	\end{center}
\end{table}

In the first experiment we study in-sample representation capacity of TT-TROM and compare it to that of 
POD--DEIM ROM. In Figure~\ref{fig:in-sample} we display the FOM, TT-TROM and POD--DEIM ROM
solutions of discretized Allen--Cahn equation at the terminal time $t = T = 20$ for two in-sample 
parameter vectors. The TT-TROM solution was computed for tensor compression accuracy $\eps=10^{-5}$.
Decreasing compression accuracy to $\eps=10^{-3}$ did not lead to a visual difference in TT-TROM solutions. 
The corresponding tensor ranks and compression factors for different values of $\eps$ are summarized in
Table~\ref{tab4}. We observe that while TT-TROM with the modest local reduced space dimensions 
$ n_{\Phi}= n_{\Psi}=20$ predicts the pattern evolution very well, the conventional POD--DEIM ROM 
is inaccurate.  
  
\begin{figure}[h]
	\begin{center}
	\begin{tabular}{ccc}\footnotesize
	FOM & TT-TROM & POD-DEIM ROM \\
		\includegraphics[width=0.305\textwidth]{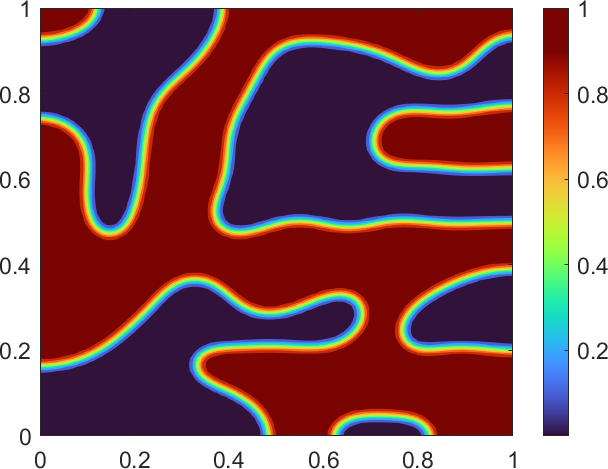} &
		\includegraphics[width=0.305\textwidth]{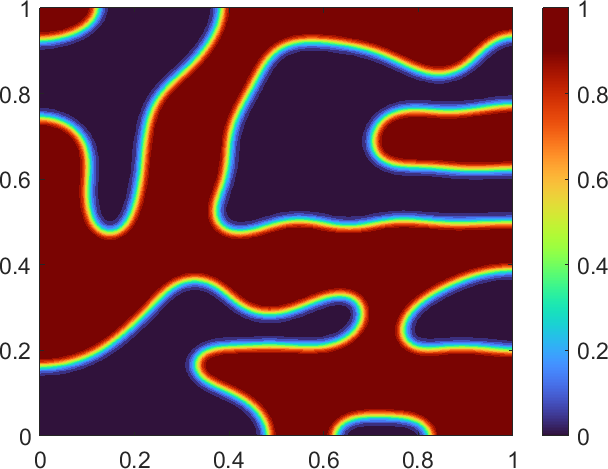} &
		\includegraphics[width=0.300\textwidth]{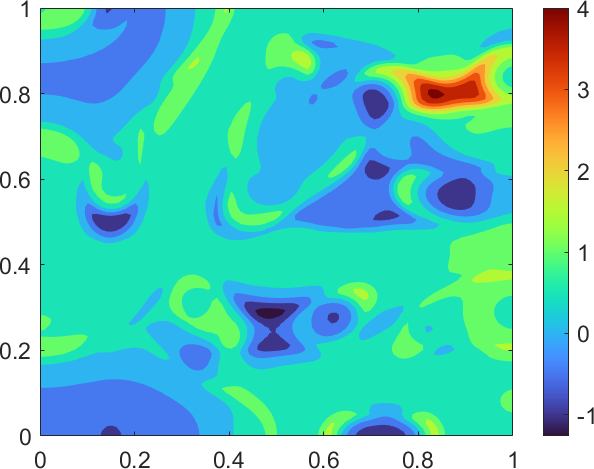} \\
		\includegraphics[width=0.305\textwidth]{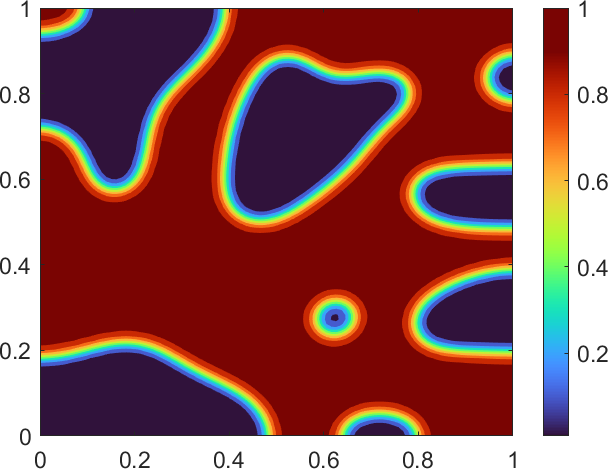} &
		\includegraphics[width=0.305\textwidth]{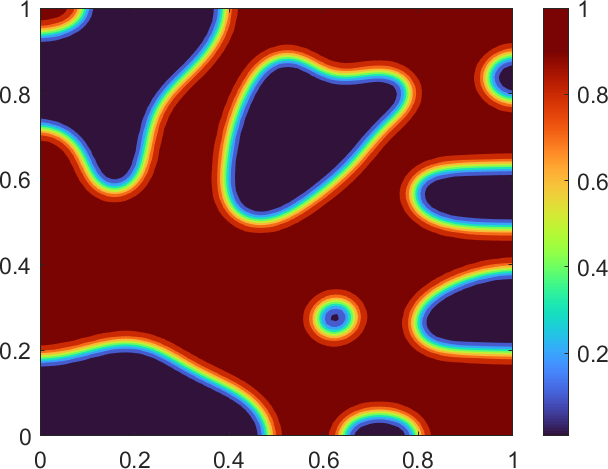} &
		\includegraphics[width=0.310\textwidth]{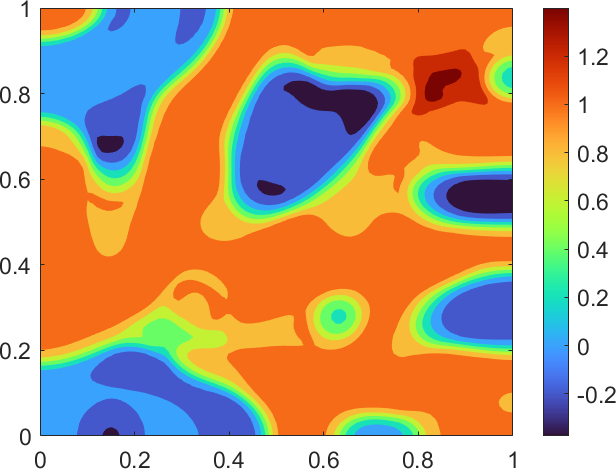} 
	\end{tabular}
	\end{center}
	\caption{FOM, TT-TROM ann POD-DEIM ROM solutions of the Allen--Cahn equation
	at $t = T = 20$ for two \textbf{in-sample} parameter vectors:
	$\balpha = (0.01, 0, 0.5)^T$ (top row) and $\balpha = (0.013, 0.15, 0.52)^T$ (bottom row).}
	\label{fig:in-sample}
\end{figure}

Next, we repeat the experiment for out-of-sample parameter values. We disply in Figure~\ref{fig:out-sample} 
the FOM, TT-TROM and POD--DEIM ROM solutions of discretized Allen--Cahn equation at the terminal time 
$t = T = 20$ for two out-of-sample parameter vectors. We observe again an excellent prediction offered by 
TT-TROM, although the relative $L^2(0,T,L^2(\Omega))$ error increased from about $10^{-5}$ for 
the in-sample case to about $10^{-2}$ for the out-of-sample case. Similarly, POD--DEIM ROM once again
fails to accurately predict the solutions.

\begin{figure}[h]
	\begin{center}\footnotesize
	\begin{tabular}{ccc}
	FOM & TT-TROM & POD-DEIM ROM \\
			\includegraphics[width=0.305\textwidth]{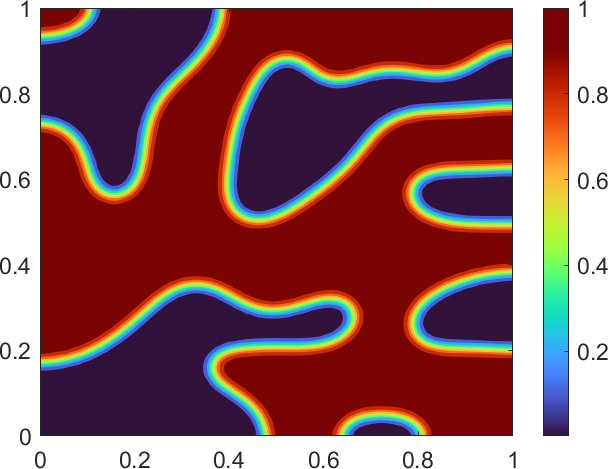} &
			\includegraphics[width=0.305\textwidth]{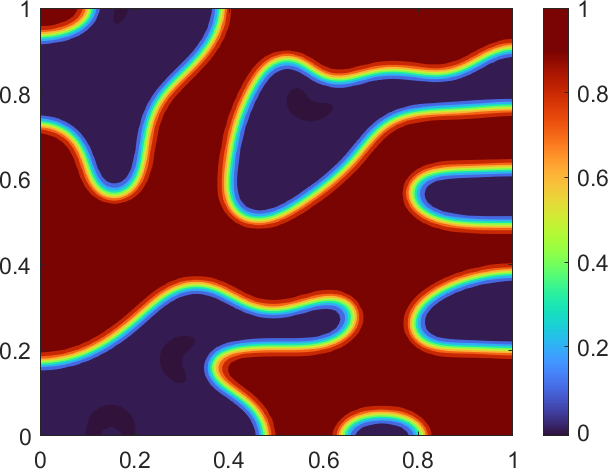} &
			\includegraphics[width=0.305\textwidth]{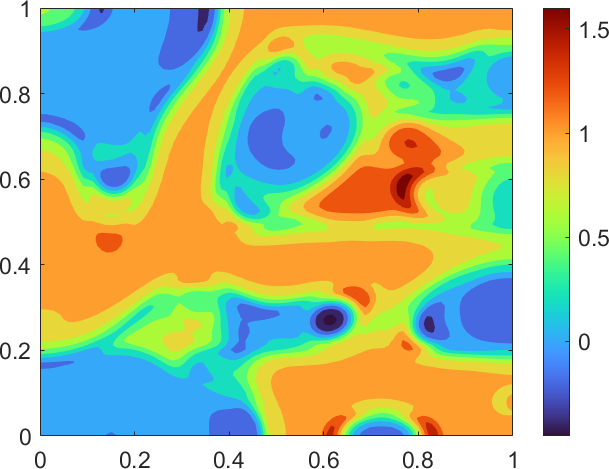} \\
			\includegraphics[width=0.305\textwidth]{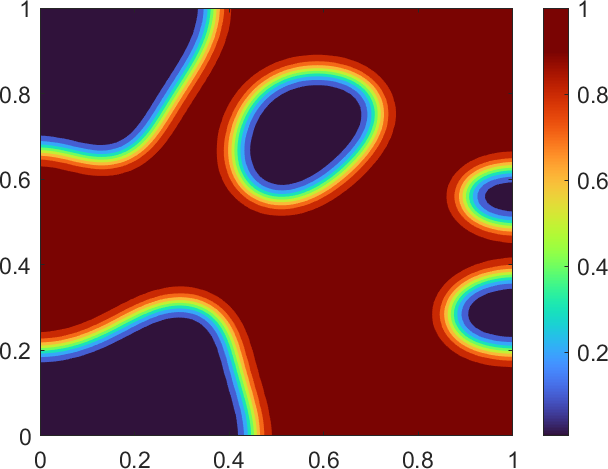} &
			\includegraphics[width=0.305\textwidth]{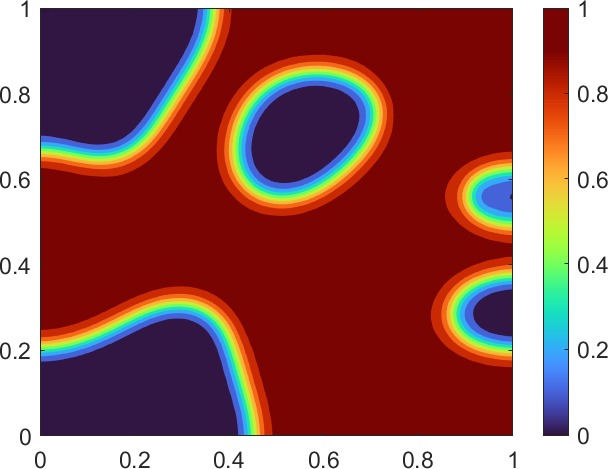} &
			\includegraphics[width=0.305\textwidth]{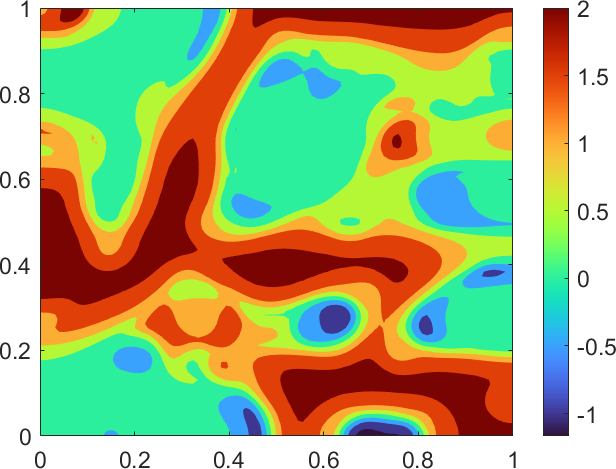}
	\end{tabular}
	\end{center}
	\caption{FOM, TT-TROM and POD-DEIM ROM solutions of the Allen--Cahn equation
	at $t = T = 20$ for two \textbf{out-of-sample} parameter vectors:
	$\balpha = (0.012, 0.1, 0.51)^T$ (top row) and $\balpha = (0.02, 0.2, 0.51)^T$ (bottom row).}
	\label{fig:out-sample}
\end{figure}

\begin{figure}[h]
	\begin{center}\footnotesize
	\begin{tabular}{ccc}
	$\eps = 10^{-1}$ & $\eps = 10^{-2}$ & $\eps = 10^{-3}$ \\
	\includegraphics[width=0.305\textwidth]{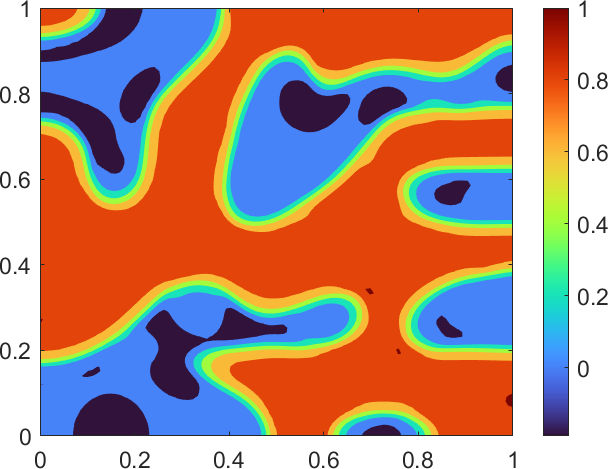} &
	\includegraphics[width=0.305\textwidth]{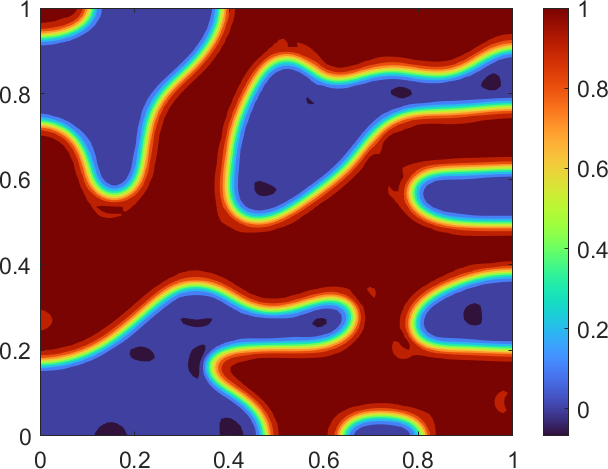} &
	\includegraphics[width=0.305\textwidth]{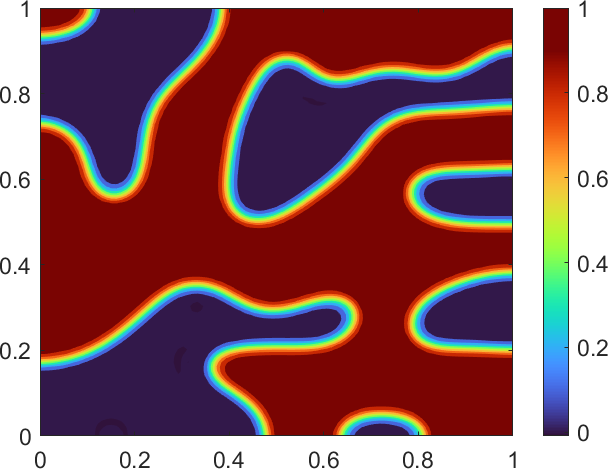}
	\end{tabular} 		
	\end{center}
	\caption{TT-TROM solutions of the Allen--Cahn equation for variying tensor compression 
	accuracy $\eps$.}
	\label{fig:TROMeps}
\end{figure}

Finally, we examine the effect of varying tensor compression accuracy $\eps$ on TROM solution accuracy and 
compression ranks. In Figure~\ref{fig:TROMeps} we display the TT-TROM solutions of discretized Allen--Cahn 
equation at $t = T = 20$ for an out-of-sample parameter vector $\balpha = (0.012, 0.1, 0.51)^T$ for the
three different values of $\eps = 10^{-1}, 10^{-2} , 10^{-3}$ with local reduced space dimensions equal to
the effective ranks, i.e., $n_{\Theta} = \widetilde{R}^\Theta_4$ for both $\Theta \in \{ \Phi, \Psi \}$.
In Table~\ref{tab4} we observe a large ratio of the first to last TT-ranks which is even higher than that
for the example in Section~\ref{sec:burgers}. This explains a vast accuracy gain offered by the TT-TROM compared 
to the conventional POD--DEIM ROM. Indeed, POD--DEIM ROM performs poorly for this example with 
solutions being inaccurate even for larger local spaces dimensions, as demonstrated in Figure~\ref{fig:POD}.

\begin{figure}[h]
	\begin{center}\footnotesize
	\begin{tabular}{ccc}
	$n_\Phi = 10, n_\Psi = 14$ & $n_\Phi = 20, n_\Psi = 35$ & $n_\Phi = 30, n_\Psi = 52$ \\
	\includegraphics[width=0.305\textwidth]{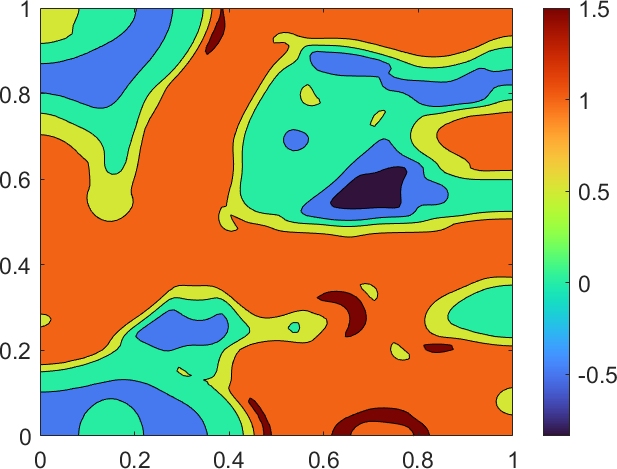} &	
	\includegraphics[width=0.305\textwidth]{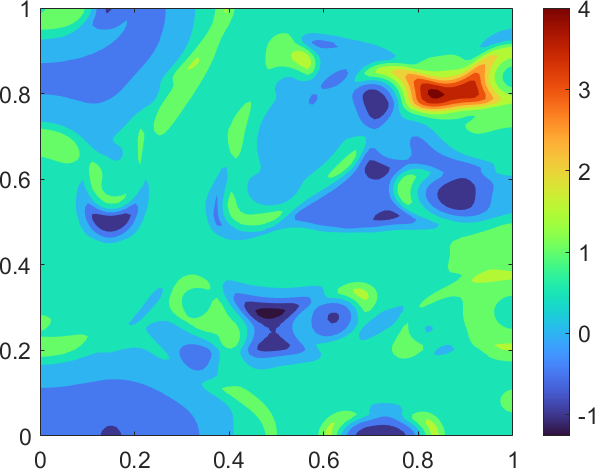} &
	\includegraphics[width=0.305\textwidth]{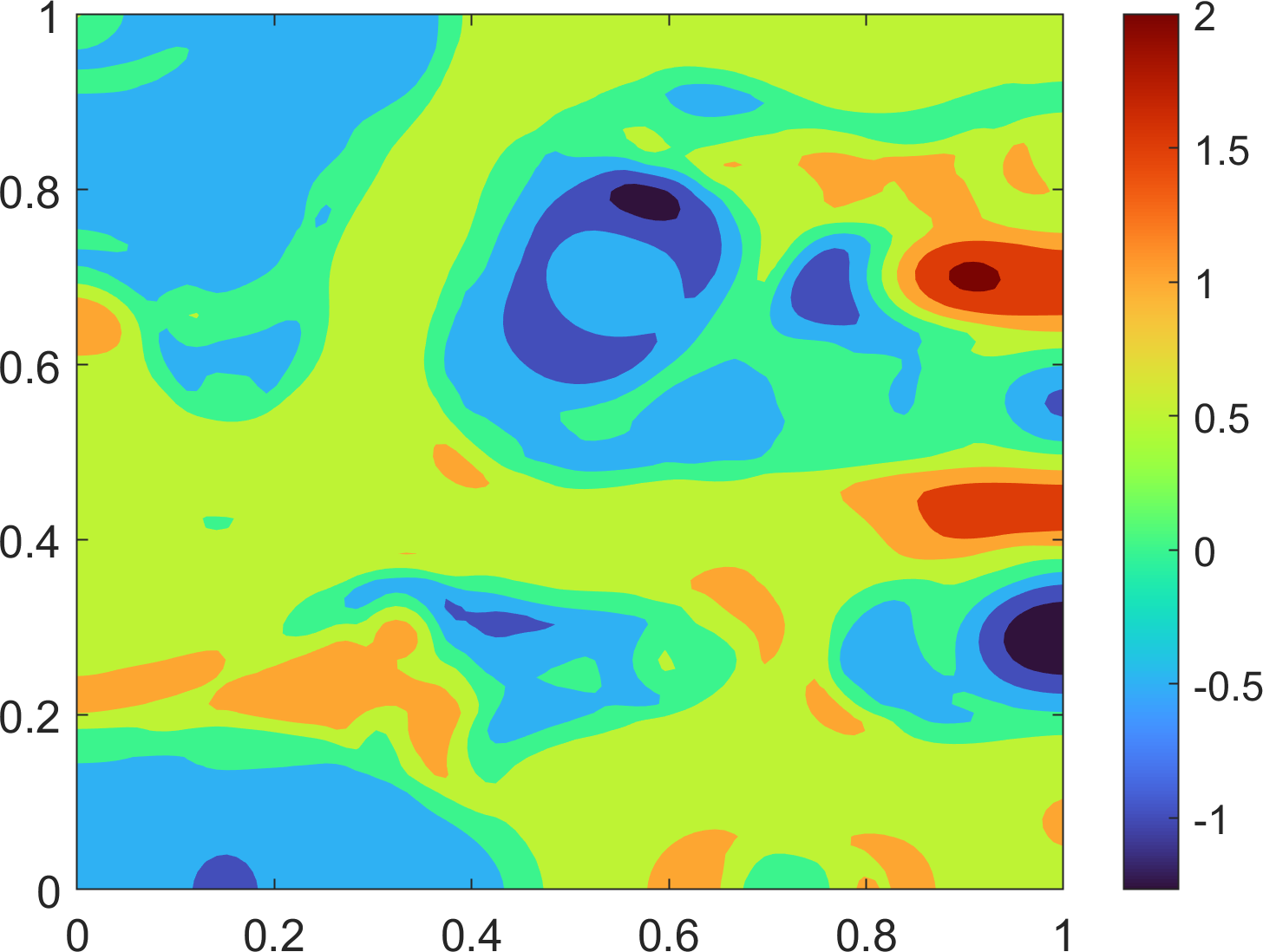} \\	
	$n_\Phi = 10, n_\Psi = 15$ & $n_\Phi = 20, n_\Psi = 35$ & $n_\Phi = 30, n_\Psi = 50$ \\
	\includegraphics[width=0.305\textwidth]{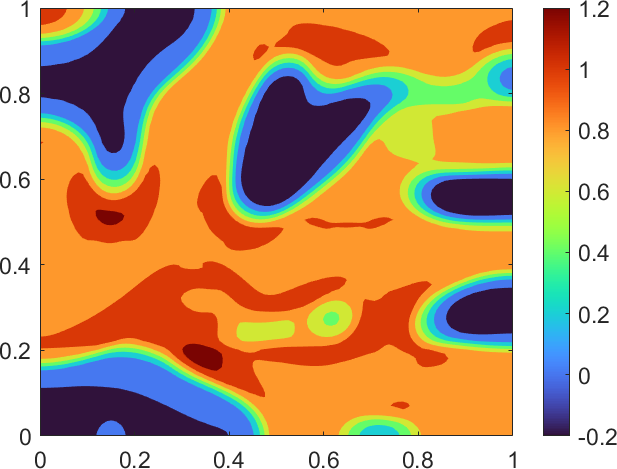} &	
	\includegraphics[width=0.305\textwidth]{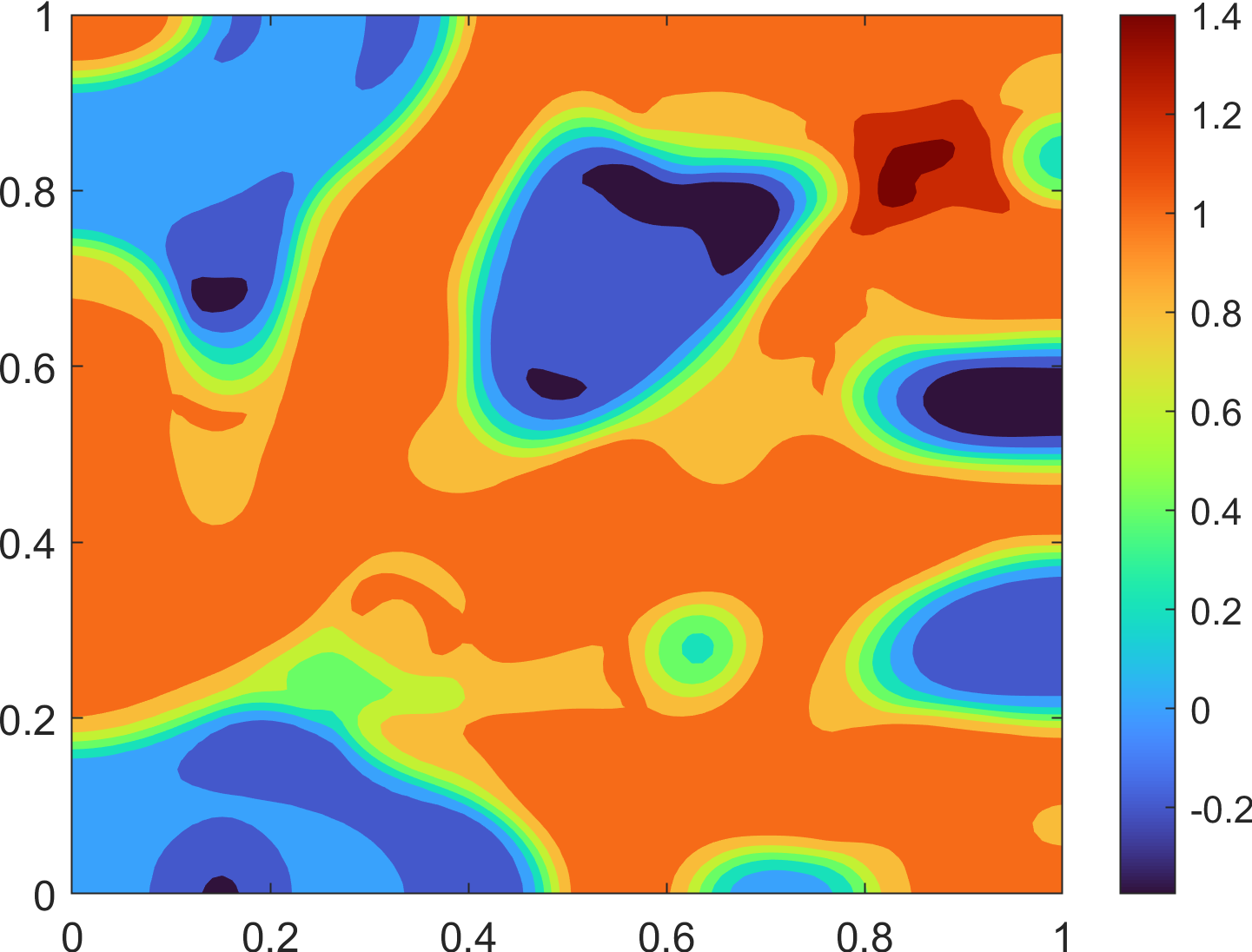} &
	\includegraphics[width=0.305\textwidth]{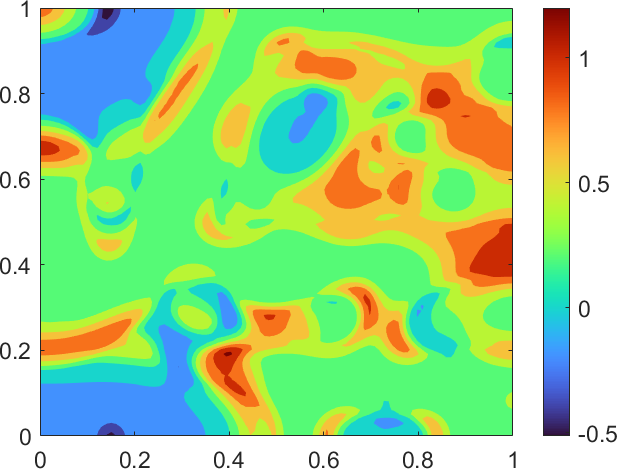} 		
	\end{tabular} 
	\end{center}
	\caption{POD-DEIM ROM solutions of the Allen--Cahn equation
	at $t = T = 20$ for two parameter vectors:
	$\balpha = (0.01, 0, 0.50)^T$ (top row) and $\balpha = (0.013, 0.15, 0.52)^T$ (bottom row) for
	increasing dimensions $n_\Phi$ and $n_\Psi$ of local reduced spaces (from left to right). 
	Some combinations of parameter values and local reduced space dimensions lead to the blow-up 
	of the numerical solution, e.g., for $\balpha = (0.013, 0.15, 0.52)^T$ and $ n_{\Phi}=10$, 
	$n_{\Psi}=15$.}
	\label{fig:POD}
\end{figure}

The computational times for the TT-TROM online stage in this Allen-Cahn equation example were mainly determined by the tensor compression ranks, as indicated in Table~\ref{tab4}. For tensor compression accuracies $\epsilon\in{10^{-2}, 10^{-3}, 10^{-4}}$, the average elapsed computational times on a laptop, calculated over multiple runs of the TT-TROM, were approximately 0.012 sec., 0.021 sec., and 0.040 sec., respectively. These times encompass steps 1 to 3 of the online stage in Algorithm~\ref{Alg1}, as well as the time required for integrating the projected system.  This can be compared to an average time of 6 sec. required by the FOM for a single value of $\balpha$.

\section{Conclusions}

We introduced a Galerkin-type model order reduction framework for non-linear parametric dynamical
systems that utilizes LRTD in place of POD for both projection and hyper-reduction steps. 
The LRTD is applied to find ``universal'' reduced spaces representing all observed snapshots 
\emph{and} it is also used for finding ``local'' parameter-specific reduced subspaces of these 
larger universal spaces. If HOSVD or TT algorithms are employed to compute LRTD of snapshot tensors, 
then the universal spaces coincide with the POD spaces for the matrices of all observed snapshots. 
In this case, the proposed TROM can be also thought of as a tensor modification of the conventional 
POD--DEIM model reduction approach that benefits from the intrinsic tensor structure of a parametric 
system in several ways: (i) it provides means to find local spaces; (ii) it allows interpolation in the parameter 
domain directly in the reduced order spaces and thus enables efficient handling of parameters  
outside of the training set; (iii) it admits a rigorous analysis of the representation power of the reduced 
spaces for general parameter values.  

We assessed the performance of three LRTD variants for model order reduction, based on CP, HOSVD and TT
tensor formats. While TT was found to have a slight edge over HOSVD in terms of compression rates for 
the examples considered, the CP variant is in general more time consuming to compute and delivers worse
approximation quality. Another variation considered is the approach to hyper-reduction in a 
two-stage setting. Out of the two variants we prefer the one with DEIM at the offline stage 
and local least squares at the online stage, since it performs similarly to the approach with DEIM on 
both stages, but also admits interpolation estimate for the local basis with a bound that is independent 
of parameters.   

We note that for large scale FOMs and higher-dimensional parameter spaces, sampling full 
snapshot tensors may become prohibitively expensive due to the exponential increase of the number 
of snapshots as a function of the paratemeter space dimension. A promising approach to decrease the 
associated offline costs is the use of low-rank tensor interpolation or completion for finding LRTD from a 
sparse sampling of parameter domain. This should decouple the required number of parameter 
samples from the dimension of parameter space, assuming a certain degree of regularity in the dependence of dynamical system solutions on the parameters. Preliminary results suggesting feasibility and efficiency of such an 
approach will be reported in a forthcoming paper.

\section*{Acknowledgments}
 M.O. was supported in  part by the U.S. National Science Foundation under awards DMS-2011444 and DMS-1953535.
 A.M. and M.O. were supported by the U.S. National Science Foundation under award DMS-2309197.
This material is based upon research supported in part by the U.S. Office of Naval Research 
under award number N00014-21-1-2370 to A.M. 

\bibliographystyle{siamplain}
\bibliography{literatur}{}

\end{document}